\documentclass[11pt]{article}

\usepackage{authblk}
\usepackage{amsmath,amsthm,amsfonts,amssymb,mathrsfs,epsfig,bm}

\usepackage[usenames]{color}
\usepackage[colorlinks=true]{hyperref}
\usepackage{graphicx}
\usepackage{marginnote}
\parskip        \medskipamount
\newtheorem{theorem}{Theorem}%[section]
\newtheorem{lemma}{Lemma}
\newtheorem{corollary}{Corollary}
\newtheorem{proposition}{Proposition}
\newtheorem{definition}{Definition}

\newtheorem{remark}{Remark}
\theoremstyle{remark}

\def\R{\mathbb{R}}

\def\P{\mathbb{P}}
\def\E{\mathbb{E}}

\def\K{\mathcal{K}}
\def\G{\mathcal{G}}

\def\H{\mathcal{H}}
\def\A{\mathcal{A}}

\def\I{\mathcal{I}}

\renewcommand{\phi}{\varphi}
\renewcommand{\epsilon}{\varepsilon}

\definecolor{mygray}{gray}{0.9}
\definecolor{deeppink}{RGB}{255,20,147}
\definecolor{mygreen}{rgb}{0.05, 0.576, 0.03}
\definecolor{myred}{rgb}{0.768, 0.09, 0.09}

\long\def\symbolfootnote[#1]#2{\begingroup
\def\thefootnote{\fnsymbol{footnote}}\footnote[#1]{#2}\endgroup}
\newcommand{\keywords}[1]{ \noindent {\footnotesize
             {\small \em Keywords and phrases.} {\sc #1} } }

\def\EE{\mathscr{E}}

\begin{document}

\title{Kingman's model with 
 random mutation probabilities: \\
convergence and condensation II}
\author{Linglong Yuan}
\affil[]{Department of Mathematical Sciences\\
University of Liverpool}

\affil[]{Department of Mathematical Sciences\\
Xi'an Jiaotong-Liverpool University}
%\email{yuanlinglongcn@gmail.com}
%\affil[]{Department of Mathematics\\ Johannes-Gutenberg-Universit\"at \ln}
%\email{yuanlinglongcn@gmail.com}
%\affil[]{Department of Mathematics\\ Uppsala University}
%\date{October 2014}
\date{\today}

\maketitle
\tableofcontents

\abstract{%Kingman's model describes the
%evolution of a one-locus haploid population of infinite size and discrete generations under the competition of selection and mutation. 
A generalisation of Kingman's model of selection and mutation has been made in a previous paper which assumes all mutation probabilities to be i.i.d.. The weak convergence of fitness distributions to a globally stable equilibrium for any initial distribution was proved. The condensation occurs if almost surely a positive proportion of the population travels to and condensates on the largest fitness value due to the dominance of selection over mutation. A criterion of condensation was given which relies on the equilibrium whose explicit expression is however unknown. This paper tackles these problems based on the discovery of a matrix representation of the random model. An explicit expression of the equilibrium is obtained and the key quantity in the condensation criterion can be estimated. Moreover we examine how the design of randomness in Kingman's model affects the fitness level of the equilibrium by comparisons between different models. The discovered facts are conjectured to hold in other more sophisticated models.

\symbolfootnote[0]{$\!\!\!$\, Address: Department of Mathematical Sciences, The University of Liverpool, Liverpool L69 7ZL, UK.

\quad Email: yuanlinglongcn@gmail.com

\quad This paper is the second half of the article submitted to Arxiv: https://arxiv.org/abs/1903.10993}
\keywords{Population dynamics, Mutation-Selection balance,  House of Cards,  Fitness distribution, Size-biased distribution, Bose-Einstein condensation,  Random matrices\\
\textit{MSC Subject classification:} 	60F05, 15B52 (primary),  60G10, 60G57, 92D15, 92D25 (secondary)
}

}

\section{Motivation}

The evolution of a population involves various forces. Kingman \cite{K78} considered the equilibrium of a population as existing because of a balance between two factors, other phenomena causing only perturbations. The pair of factors he chose was mutation and selection. The most famous model for the evolution of one-locus haploid population of infinite size and discrete generations, proposed by Kingman \cite{K78}, is  as follows: 

Let the fitness value of any individual take values in $[0,1].$ Higher fitness values represent higher productivities. Let $(P_n)=(P_n)_{n\geq 0}$ be a sequence of probability measures on $[0,1],$ and denote the fitness distribution of the population at generation $n$. Let $b\in[0,1)$ be a mutation probability. Let $Q$ be a probability measure on $[0,1]$ serving as mutant fitness distribution. Then $(P_n)$ is constructed by the following iteration:

\begin{equation}\label{kpi} P_{n}(dx)=(1-b)\frac{x P_{n-1}(dx)}{\int y P_{n-1}(dy)}+bQ(dx), \quad n\geq1.\end{equation}
Biologically it says that a proportion $b$ of the population are mutated with fitness values sampled from $Q$ and the rest will undergo the selection via a size-biased transformation.  Kingman used the term ``House of Cards'' for the fact that the fitness value of a mutant is independent of that before mutation, as the mutation destroys the biochemical ``house of cards'' built up by evolution. 

House-of-Cards models, which includes Kingman's model, belong to a larger class of models on the balance of mutation and selection. Variations and generalisations of Kingman's model have been proposed and studied for different biological purposes, see for instance B\"urger \cite{B86,B89, B98,B00}, Steinsaltz {\it et al} \cite{SEW05}, Evans {\it et al} \cite{SEW13} and Yuan \cite{Y15}. We refer to  \cite{Y20} for a more detailed literature review. 

But to my best knowledge, no random generalisation has been developed except in my previous paper \cite{Y20}, in which we assume that the mutation probabilities form an i.i.d.\ sequence. The randomness of the mutation probabilities reflects the influence of a stable random environment on the mutation mechanism. The fitness distributions have been shown to converge weakly to a globally stable equilibrium distribution for any initial fitness distribution. When selection is more favoured than mutation, a condensation may occur, which means that almost surely a positive proportion of the population travels to and condensates on the largest fitness value. We have obtained a criterion of condensation  which relies on the equilibrium whose explicit expression is however unknown. So we do not know how the equilibrium looks like and whether condensation occurs or not in concrete cases. 

As a continuation for \cite{Y20}, this paper aims to solve the above problems based on the discovery of a matrix representation of the random model which yields an explicit expression for the equilibrium.  The matrix representation also allows to examine the effects of different designs of randomness by comparing the moments and condensation sizes of the equilibriums in several models.

\section{Models}
This section is mainly a summarisation of Section 2 in \cite{Y20}, in addition to the  introduction of a new random model where all mutation probabilities are equal but random. 
\subsection{Two deterministic models}
Let $M_1$ be the space of probability measures on $[0,1]$ endowed with the topology of weak convergence. Let $(b_n)=(b_n)_{n\geq 1}$ be a sequence of numbers in $[0,1)$, and $P_0, Q\in M_1$. {\it Kingman's model with time-varying mutation probabilities} or simply {\it the general model} has parameters $(b_n), Q, P_0$. 
In this model, $(P_n)=(P_n)_{n\geq 0}$ is a (forward) sequence of probability measures in $M_1$
generated by 

\begin{equation}\label{pi} P_{n}(dx)=(1-b_{n})\frac{x P_{n-1}(dx)}{\int y P_{n-1}(dy)}+b_{n}Q(dx), \quad n\geq1,\end{equation}
where $\int$ denotes $\int_0^1 .$ We introduce a function $S:M_1\mapsto [0,1]$ such that 
$$S_u:=\sup\{x:u([x,1])>0\},\quad \forall u\in M_1.$$
Then $S_u$ is interpreted as the {\it largest fitness value} of a population of distribution $u$. Let $h:=S_{P_0}$ and assume that $h\geq S_Q.$ This assumption is natural because in any case we have 
$S_{P_1}\geq S_Q.$

We are interested in the convergence of $(P_n)$ to a possible equilibrium, which is however not guaranteed without putting appropriate conditions on $(b_n)$. To avoid triviality, we do not consider $Q=\delta_0$, the dirac measure on $0$.

 {\it Kingman's model }is simply the model when $b_n=b$ for any $n$ with the parameter $b\in[0,1)$. We say a sequence of
probability measures $(u_n)$ converges in total variation to $u$ if the total variation $\|u_n-u\|$ converges to zero.  It was shown by Kingman \cite{K78} that $(P_n)$ converges to a probability measure, that we denote by $\K$, which depends only on $b,Q$ and $h$ but not on $P_0.$ 

\begin{theorem}[Kingman's theorem,\cite{K78}]\label{King}
If  $\int \frac{Q(dx)}{1-x/h}\geq b^{-1},$ then $(P_n)$ converges in total variation  to 
$$\K(dx)=\frac{b \theta_bQ(dx)}{\theta_b-(1-b)x},$$
where $\theta_b$, as a function of $b$, is the unique solution of 
\begin{equation}\label{sb}\int\frac{b \theta_bQ(dx)}{\theta_b-(1-b)x}=1.\end{equation}
If $\int \frac{Q(dx)}{1-x/h}< b^{-1}$, then $(P_n)$ converges weakly to 
$$\K(dx)=\frac{b Q(dx)}{1-x/h}+\Big(1-\int\frac{b Q(dy)}{1-y/h}\Big)\delta_{h}(dx).$$ 
%here $\delta_h(dx)$ is the Dirac measure at $h.$ 
\end{theorem}
We say {\it there is a condensation on $h$ in Kingman's model} if $Q(h)=Q(\{h\})=0$ but $\K(h)>0$, which corresponds to the second case above. We call $\K(h)$ {\it the condensate size on $h$ in Kingman's model} if $Q(h)=0$. The terminology is due to the fact that if we let additionally $P_0(h)=0,$ then any $P_n$ has no mass on the extreme point $h$; however asymptotically a certain amount of mass $\K(h)$ will travel to and condensate on $h$. 

\subsection{Two random models}
We recall the notation of weak convergence for random probability measures. Let $(\mu_n)$ be random probability measures supported on $[0,1]$. The sequence converges weakly to a limit $\mu$ if and only if for any continuous function $f$ on $[0,1]$ we have 
$$\int f(x)\mu_n(dx)\stackrel{d}{\longrightarrow}\int f(x)\mu(dx).$$

Next we introduce two random models which generalise Kingman's model.  Let $\beta\in[0,1)$ be a random variable. Let $(\beta_n)$ be a sequence of i.i.d.\ random variables sampled from the distribution of $\beta.$   If $b_n=\beta_n$ for any $n$ we call it {\it Kingman's model with random mutation probabilities} or simply {\it the first random model}. It has been proved in \cite{Y20} that $(P_n)$ converges weakly to a globally stable equilibrium, that we denote by $\I$ whose distribution depends on $\beta, Q, h$ but not on $P_0$.

For comparison we introduce another random model. If $b_n=\beta$ for any $n$, we call it {\it Kingman's model with the same random mutation probability} or {\it the second random model}. Conditionally on the value of $\beta$, it becomes Kingman's model. So we can think of this model as a compound version of Kingman's model, with $b$ replaced by $\beta.$ We denote the limit of $(P_n)$ by $\A$ which is a compound version of $\K$.

In this paper, we continue to study the equilibrium and the condensation phenomenon in the first random model. 
By Corollary 4 in \cite{Y20}, if $Q(h)=0$, then $\I(h)>0$ a.s. or $\I(h)=0$ a.s.. We say {\it there is a condensation on $h$ in the first random  model} if $Q(h)=0$ but $\I(h)>0$ a.s.. We call $\I(h)$ {\it the condensate size on $h$} if $Q(h)=0$. A condensation criterion, which relies on a function of $\beta$ and $\I$, was established in \cite{Y20}. As the equilibrium has no explicit expression, the condensation criterion cannot be used in concrete cases. This paper aims to solve these problems based on a matrix representation of the general model which can be inherited to the first random model. The objectives include an explicit expression of $\I$, and finer properties of $\I$ on the moments and condensation. The comparison of Kingman's model and the two random models will be performed and to this purpose we assume additionally that 
$$\E[\beta_n]=\E[\beta]=b\in (0,1),\quad \forall \,n\geq 1\,\,.$$
The case with $b=0$ is excluded for triviality.

\section{Notations and results}
\subsection{Preliminary results}\label{pre}
In  this section, we again recall some necessary results from \cite{Y20}. 
We introduce 
$$Q^k(dx):=\frac{ x^kQ(dx)}{\int y^kQ(dy)},\quad m_k:=\int x^kQ(dx),  \quad \forall\, k\geq 0.$$

We introduce the notion of {\it invariant measure}. A random measure $\nu\in M_1$ is {\it invariant}, if it satisfies 
$$\nu(dx)\stackrel{d}{=}(1-\beta)\frac{x \nu(dx)}{\int_0^1 y \nu(dy)}+\beta Q(dx)$$
with $\beta$ independent of $\nu$. Note that $\I$, the limit of $(P_n)$ in the first random model, is an invariant measure. 

In the general model a forward sequence $(P_n)$ does not necessarily converge.
But the convergence may hold if we investigate the model in a backward way. A finite backward sequence $( P_j^n)=( P_j^n)_{0\leq j\leq n}$ has parameters $n, (b_j)_{1\leq j\leq n},Q, P_n^n, h$ with $h=S_{P_n^n}$ and satisfies 
\begin{equation}\label{revp} P_j^n(dx)=(1-b_{j+1})\frac{x P_{j+1}^n(dx)}{\int y P_{j+1}^n(dy)}+b_{j+1}Q(dx), \quad 0\leq j\leq n-1.\end{equation}
Consider a particular case with $P_n^n=\delta_h$. Then $P_j^n$ converges in total variation to a limit, denoted by $\G_j=\G_{j,h}$ (and $\G=\G_0, \G_Q=\G_{0,S_Q}$), as $n$ goes to infinity with $j$ fixed, such that 
\begin{equation}\label{gf} \G_{j-1}(dx)=(1-b_{j})\frac{x \G_j(dx)}{\int y \G_j(dy)}+b_{j}Q(dx), \quad j\geq 1\end{equation}
where $\G:[0,1)^\infty\to M_1$ is a measurable function, with $\G_j=\G(b_{j+1},b_{j+2,\cdots})$ which is supported on $[0,S_Q]\cup \{h\}$ for any $j$. Moreover, $(\ref{gf})$ can be further developed
\begin{equation}\label{bsfas}\G_0(dx){=}G_0\delta_h(dx)+\sum_{j=0}^{\infty}\prod_{l=1}^{j}\frac{(1-b_l)}{\int y\G_l(dy)}b_{j+1}m_jQ^j(dx). \end{equation}
where $G_0=G_{0,h}=1-\sum_{j=0}^{\infty}\prod_{l=1}^{j}\frac{(1-b_l)}{\int y\G_l(dy)}b_{j+1}m_j.$
Then $\G_0$ can be considered as a convex combination of probability measures $\{\delta_h, Q,Q^1,Q^2,\cdots\}$. We introduce also $G_j=G_{j,h}$ for $\G_{j,h}$ for any $j$
and $G=G_0, G_Q=G_{0,S_Q}.$

The above results hold regardless of the values of $(b_n)$. So they hold also in the other three models. In particular, we replace the symbol $\G, G$ by $\I, I$ in the first random model, by $\A, A$ in the second random model and by $\K, K$ in Kingman's model.%, similarly for all the results from the general model in this paper. 

For the first random model, $(\I_j)$ is stationary ergodic and $\I$ is the weak limit of $(P_n)$. Moreover $\E\left[\ln \frac{(1-\beta)}{\int y\I_Q(dy)}\right]\in [-\infty,-\ln\int yQ(dy)]$ is well defined, whose value does not depend on the joint law of $(\beta, \I)$. This term is the key quantity in the condensation criterion. Note that we neither have an explicit expression of $\I_Q$ nor an estimation of $\E\left[\ln \frac{(1-\beta)}{\int y\I_Q(dy)}\right].$%to tell whether $I_{0,h}$ equals zero or not: 
\begin{theorem}[Condensation criterion, Theorem 3 in \cite{Y20}]\label{conY20}
\begin{enumerate}
\item If $h=S_Q$, then there is no condensation on $S_Q$ if 
\begin{equation}\label{uq}\E\left[\ln \frac{S_Q(1-\beta)}{\int y\I_Q(dy)}\right]<0.\end{equation}
\item If $h>S_Q$, then there is no condensation on $h$ if and only if 
\begin{equation}\label{h}\E\left[\ln \frac{h(1-\beta)}{\int y\I_Q(dy)}\right]\leq 0.\end{equation}
\end{enumerate}
\end{theorem}

%Note that the condition (\ref{uq}) implies $Q(S_Q)=0.$ So it is legitimate to study condensation. 
%(\ref{uq}) and (\ref{h}) serve as condensation criterion in respective cases.

\subsection{Notations on matrices}
The most important tool in this paper is the matrix representation in the general model. We need to firstly introduce some notations and functions related to matrix. 
One can skip this part at first reading. 

{\bf \noindent 1).} Define
$$\gamma_j=\frac{1-b_j}{b_j},\quad  \gamma=\frac{1-b}{b},\quad \Gamma_j=\frac{1-\beta_j}{\beta_j},\quad \Gamma=\frac{1-\beta}{\beta}$$
where the 4 terms all belong to $(0,\infty].$
For any $\, 1\leq j\leq n\leq \infty$ (except $j=n=\infty$), define 
\begin{equation}\label{ajnx}W_x^{j,n}:=\left(\begin{array}{ccccc}x&x^2&x^3&\cdots&x^{n-j+2}\\
-\gamma_j&m_1&m_2&\cdots&m_{n-j+1}\\
0&-\gamma_{j+1}&m_1&\cdots&\vdots\\
0&0&\ddots&\ddots&\vdots\\
0&0&\cdots&-\gamma_{n}&m_1\end{array}\right)%(W^{j,n}_1,W^{j,n}_2,\cdots,W^{j,n}_{n-j+2})
,\end{equation}
and 
\begin{equation}\label{ajn}W^{j,n}:=\int W_x^{j,n}Q(dx)=\left(\begin{array}{ccccc}m_1&m_2&m_3&\cdots&m_{n-j+2}\\
-\gamma_j&m_1&m_2&\cdots&m_{n-j+1}\\
0&-\gamma_{j+1}&m_1&\cdots&\vdots\\
0&0&\ddots&\ddots&\vdots\\
0&0&\cdots&-\gamma_{n}&m_1\end{array}\right)%(W^{j,n}_1,W^{j,n}_2,\cdots,W^{j,n}_{n-j+2})
.\end{equation}
Introduce $$W_x^n=W_x^{1,n};\,\,\,\, \,\,W_x=W_x^{1,\infty};\,\,\,\,\, \,W_x^{n+1,n}=(x);\,\, \,\,\,\,W_x^{m,n}=(1),\forall m>n+1$$ 
and $$W^n=W^{1,n}; \,\,\,\,\,\, W=W^{1,\infty}; \,\,\,\,\,\, W^{n+1,n}=(m_1); \,\,\,\,\,\, W^{m,n}=(1),\forall m>n+1.$$
%For $1\leq k\leq n$, let $B_x^{k,n}$ be the matrix of $W_x^n$ with the first row replaced by 
%$$(0,0,\cdots,0,x^{k+1},x^{k+2},\cdots, x^{n+1})$$
%and let $B^{k,n}=\intB_x^{k,n}Q(dx).$
{\bf \noindent 2).} For a matrix $M$ of size $m\times n$, let 
$r_i(M)$ be the $i$th row and $c_j(M)$ be the $j$th column, for $1\leq i\leq m, 1\leq j\leq n$. 
If the matrix is like
$$M=\left(\begin{array}{ccccc}m_{a_1}&m_{a_2}&\cdots & m_{a_{n-1}}&m_{a_n}\\
\cdot&\cdot&\cdots&\cdot&m_{a_{n+1}}\\
\vdots&\vdots&\ddots&\vdots&\vdots\\
\cdot&\cdot&\cdots&\cdot&m_{a_{n+m}}
\end{array}\right),$$ 
define, for any $k\geq 0$ 
$$U_k^rM:=\left(\begin{array}{ccccc}m_{k+a_1}&m_{k+a_2}&\cdots & m_{k+a_{n-1}}&m_{k+a_n}\\
\cdot&\cdot&\cdots&\cdot&m_{a_{n+1}}\\
\vdots&\vdots&\ddots&\vdots&\vdots\\
\cdot&\cdot&\cdots&\cdot&m_{a_{n+m}}
\end{array}\right).$$
Here $U_k^r$ increases the indices of the first row by $k$, with $r$ referring to ``row", and $U$ to ``upgrade". 
Similarly define 
$$U_k^cM:=\left(\begin{array}{ccccc}m_{a_1}&m_{a_2}&\cdots & m_{a_{n-1}}&m_{k+a_n}\\
\cdot&\cdot&\cdots&\cdot&m_{k+a_{n+1}}\\
\vdots&\vdots&\ddots&\vdots&\vdots\\
\cdot&\cdot&\cdots&\cdot&m_{k+a_{n+m}}
\end{array}\right)$$
which increases the indices of the last column by $k$, with $c$ referring to ``column".  In particular we write 
$$U^r=U_1^r, \quad U^c=U_1^c.$$

{\bf \noindent 3). }
Let $|\cdot|$ denote the determinant operator for square matrices. It is easy to see that, if none of $\gamma_j, \gamma_{j+1}, \cdots, \gamma_n$ is equal to infinity,  
$$|U_k^rW^{j,n}|>0,\,\, |U_k^cW^{j,n}|>0, \quad \forall\, k\geq 0, 1\leq j\leq n+1.$$
Define 
\begin{equation}\label{lr}L_{j,n}:=\frac{|W^{j+1,n}|}{|W^{j,n}|},\,\,\,\, R_{j,k}^n:=\frac{|U_k^rW^{j,n}|}{|W^{j,n}|},\,\,\,\,R_{j}^n:=R_{j,1}^n, \quad  \forall \,1\leq j\leq n,\, k\geq 1.\end{equation}
Specifically, let $L_{n+1,n}=\frac{1}{m_1}, R^n_{n+1,k}=\frac{m_{k+1}}{m_1}$. In the above definition, if one or some of $\gamma_j, \gamma_{j+1}, \cdots, \gamma_n$ are infinite, we consider $L_{j,n},R_{j,k}^n$ as obtained by letting the concerned variables go to infinity. As a convention, we will not mention again the issue of some $\gamma_j$'s being infinite, when the function can be defined at infinity by limit.  

Notice that expanding $W^{j,n}$ along the first column, we have 
\begin{equation}\label{laa}L_{j,n}=\frac{|W^{j+1,n}|}{|W^{j,n}|}=\frac{|W^{j+1,n}|}{m_1|W^{j+1,n} |+\gamma_j |U^rW^{j+1,n}|}=\frac{1}{m_1+\gamma_jR^n_{j+1}}.\end{equation}
If $\gamma_j=\infty$, let 
$$L_{j,n}=0,\,\, \gamma_jL_{j,n}=\frac{1}{R_{j+1}^n}. $$
%By Remark \ref{llimit}, $L_{i,n}$ decreases strictly to a limit $L_i$ as $n$ tends to infinity if $b_i\neq 0$. Otherwise $L_{i,n}=0$ for any $n\geq i.$

\begin{lemma}\label{rl}
In the general model, $R_{j,k}^n$ increases strictly in $n$ to a limit that we denote by $R_{j,k}$ (and $R_j=R_{j,1}$) which satisfies 
\begin{equation}\label{rik}
R_{j,k}=\frac{m_{k+1}+\gamma_iR_{j+1,k+1}}{m_1+\gamma_iR_{j+1}}.\end{equation}
And $\gamma_j L_{j,n}$ decreases strictly in $n$ to a limit that we denote by $\gamma_j L_{j}$ which satisfies

\begin{equation}  \label{rjlj}
\gamma_jL_{j}=\left\{  
             \begin{array}{lr}  
            \Large{1/{R_{j+1}}}, &  \mbox{if $\gamma_j=\infty$};\\  
             & \\  
              \gamma_j/({m_1+\gamma_jR_{j+1}}), &    \mbox{if $\gamma_j<\infty$}.
             \end{array}  
\right.  
\end{equation}  
Moreover 
\begin{equation}\label{boundrjlj}\frac{\gamma_j}{m_1+\gamma_j}< \gamma_j L_j< \frac{\gamma_j}{m_1(1+ \gamma_j)}, \quad m_1<R_{j+1}<1.\end{equation}
\end{lemma}

\subsection{Main results}

{\bf \noindent 1). Matrix representation. }

\vspace{3 mm}

We set a convention that for a term, say $\alpha_j$, in the general model, we use $\widetilde \alpha_j$ to denote the corresponding term in the first random model and $\widehat \alpha_j$ in the second random model, $\overline \alpha_j$ in Kingman's model. If the corresponding term does not depend on the index $j$, we just omit the index. 

Consider a finite backward sequence $(P_j^n)$ in the general model: 
\begin{equation}\label{recursion}P_n^n=Q, \quad P_j^n(dx)=(1-b_{j+1})\frac{xP_{j+1}^n(dx)}{\int yP_{j+1}^n(dy)}+b_{j+1}Q(dx), \quad   \,0\leq j\leq n-1.\end{equation}
The previous sequence used in Section \ref{pre} starts with $P_n^n=\delta_h$ and this one starts with $P_n^n=Q$. 
The advantage of this change is that the latter enjoys a matrix representation, which is the most important tool in this paper.
\begin{lemma}\label{matrix}Consider $(P_j^n)$ in (\ref{recursion}). 
For any $0\leq j\leq n$,\begin{equation}\label{p0n}
\frac{xP_{j}^n(dx)}{\int yP_j^n(dy)}=\frac{|W_x^{j+1,n}|}{|W^{j+1,n}|}Q(dx),
\end{equation}
and 
\begin{equation}\label{pphi}P_j^n(dx)=(1-b_{j+1})\frac{|W_x^{j+2,n}|}{|W^{j+2,n}|}Q(dx)
+b_{j+1}Q(dx).\end{equation}
\end{lemma}
Letting $n$ go to infinity, we obtain the following. 
\begin{theorem}\label{h}
For $j$ fixed and $n$ tending to infinity, $P_j^n$ converges weakly to a limit, denoted by $\H_j$. If we denote $\H=\H_0$, then $\H:[0,1)^\infty\to M_1$ is a measurable function such that 
\begin{equation}\label{hjh}\H_j=\H(b_{j+1},b_{j+2},\cdots),\end{equation}
and
\begin{equation}\label{hj+1}\H_{j}(dx)=(1-b_{j+1})\frac{x \H_{j+1}(dx)}{\int y \H_{j+1}(dy)}+b_{j+1}Q(dx).\end{equation}
Moreover 
\begin{equation}\label{hl}\frac{1-b_{j+1}}{\int y \H_{j}(dy)}=\gamma_{j+1}L_{j+1}.\end{equation}
\end{theorem}
Note that $(\H_j)$ is the limit of $(P_j^n)$ with $P_n^n=Q$, and $(\G_j)$ is the limit of $(P_j^n)$ with $P_n^n=\delta_h$. When $h=S_Q,$ it remains open whether $\H=\G_Q$ or not. But the equality holds in the first random model.

\begin{corollary}\label{h=g}
It holds that
$$(\I_{j,S_Q})\stackrel{d}{=}\left(\widetilde\H_j\right).$$

\end{corollary}

\vspace{3 mm}

{\bf \noindent 2). Condensation criterion. }

\vspace{3 mm}

A remarkable application of the matrix representation is that 
the condensation criterion in Theorem \ref{conY20} can be written into a simpler and tractable form using matrices.

\begin{corollary}[Condensation criterion]\label{con}
\begin{enumerate}
\item If $h=S_Q$, then there is no condensation on $\{S_Q\}$ if 
\begin{equation}\E\left[\ln S_Q\Gamma_{1}\widetilde L_{1}\right]<0.\end{equation}
\item If $h>S_Q$, then there is no condensation on $\{S_Q\}$ if and only if
\begin{equation}\E\left[\ln h\Gamma_{1}\widetilde L_{1}\right]\leq 0.\end{equation}
\end{enumerate}
\end{corollary}

Note that the key quantity $\E\left[\ln \frac{(1-\beta)}{\int y\I_Q(dy)}\right]$ in Theorem \ref{conY20} is now rewritten as $\E\left[\ln \Gamma_{1}\widetilde L_{1}\right]$. An estimation of it is highly necessary to make the criterion applicable. To achieve this, we introduce the second important tool of this paper in the following lemma, which is interesting by itself. 
\begin{lemma}\label{key}
Let $f(x_1,\cdots,x_n):\R^n\mapsto\R$ be a bounded $C^2$  function with $\sum_{1\leq i\neq j \leq n}f_{x_ix_j}\leq 0$. Let $(\xi_1,\cdots,\xi_n)$ 
be $n$ exchangeable random variables in $\R$. Then 
$$\E[f(\xi_1,\cdots\xi_n)]\geq \E[f(\xi_1,\cdots,\xi_1)].$$ 
\end{lemma}

The estimation of $\E[\ln \Gamma_1\widetilde L_1]$ is given as follows. 
\begin{theorem}\label{gammal3}
%With $\gamma_\beta,\gamma, L_\beta, l$ defined as above, 
We have 
\begin{equation}\label{3}\E[\ln  \Gamma \widehat L]\leq \E[\ln \Gamma_1\widetilde L_1]\leq \ln \gamma\overline L\end{equation}
where \begin{equation}\label{bargammal}\gamma \overline L=\frac{1-b}{\int y \K_Q(dy)}=\left\{ \begin{array}{ll}
          \frac{1-b}{\theta_b}, & \mbox{if $\int \frac{Q(dx)}{1-x/S_Q}> b^{-1}$};\\
&\\
         \frac{1}{S_Q}, & \mbox{if $\int \frac{Q(dx)}{1-x/S_Q}\leq  b^{-1}$},\end{array} \right.\end{equation}
and 
$$\Gamma \widehat L= \frac{1-\beta}{\int y\A_Q(dy)}.$$
\end{theorem}
\begin{remark}
The two inequalities in (\ref{3}) are not strict in general. Here is an example.  By Theorem \ref{King}, if $\int \frac{Q(dx)}{1-x/S_Q}\leq b^{-1},$ one can obtain by simple computations that $\gamma \overline L=1/S_Q$. For the same reason, if $\int \frac{Q(dx)}{1-x/S_Q}\leq \beta^{-1}$ almost surely, then $\Gamma \widehat L=1/S_Q$ almost surely. So taking $\beta$ and $b$ small enough, the two inequalities in (\ref{3}) become equalities. %To know when the inequalities occur more efforts are needed to exploit the matrix approach. 
\end{remark}

As Kingman's model is a special kind of the first random model, Corollary \ref{con} applies to Kingman's model as well. The second inequality in (\ref{3}) implies that Kingman's model is easier to have condensation than the first random model in general. This is made more clear in the next Theorem \ref{randet}. 
\vspace{3 mm}

{\bf \noindent 3). Comparison between the first random model and the other models. }

\vspace{3 mm}

For succinctness, the results that we present in this part are only in the case $h=S_Q$. 
However all the results can be easily proved for $h>S_Q$, if we do not stick with strict inequalities. The main idea is to take a new mutant distribution $(1-\frac{1}{n})Q+\frac{1}{n}\delta_h$ and consider the limits of equilibriums as $n$ tends to infinity. 

We consider an equilibrium to be fitter if it has higher moments and bigger condensate size. In the following, we provide three theorems on the comparison of moments and/or condensate sizes. 
\begin{theorem}\label{randet}
Between Kingman's model and the first random model, if $\P(\beta=b)<1$, we have
\begin{enumerate}
\item in terms of moments, $$\E\left[\int y^k\I_Q(dy)\right]<  \int y^k \K_Q(dy), \quad \forall \,k=1, 2,\cdots.$$
\item in terms of condensate size, if $Q(S_Q)=0$ and $I_Q>0, a.s.,$ then
$$\E[I_Q]<  K_Q.$$
%\item if $Q(u_Q)>0$ (hence $C_*^{u_Q}$=0, a.s.), then
%$$\E[P_*^{u_Q}(u_Q)]< \bar P_*^{u_Q}(u_Q).$$
\end{enumerate}
 \end{theorem}
%\begin{remark}The above comparison shows that Kingman's model overwhelmingly domin Ates random i-i-d model in terms of condensation size and fitness. 
%\end{remark}
%he above comparison shows that the first random model is completely dominated by Kingman's model in terms of condensation and fitness.

\begin{theorem}\label{iidofa} Between the two random models, the following inequality holds
$$\E\left[\ln \int y\I_Q(dy)\right]\leq \E\left[\ln \int y\A_Q(dy)\right].$$
\end{theorem}
%So the second random model yields an equilibrium fitter than that of the first random model in terms of the logarithm of the moment of order 1. 

%\noindent Conjecture: $\E[ \int y^kdP_*^h]\leq \E[ \int y^kd\widehat P_*^h], \quad \forall k\geq 1$

\begin{theorem}\label{kinofa}Between Kingman's model and the second random model, it holds that
$$\E[A_Q]\geq K_Q, \text{ if }Q(S_Q)=0.$$
But there is no one-way inequality between $\E[\int y\A_Q(dy)]$ and $ \int y\K_Q(dy)$.
%For any $k\geq 1$, there is no one-way inequality between $\E[\int y^k d\widehat P_*^h]$ and $ \E[\int y^kP_*^h]$ nor between $\E[\ln \int y^k d\widehat P_*^h]$ and $ \E[\ln\int y^k P_*^h]$. 
\end{theorem}
%It is more delicate between the two models. Kingman's model has a smaller condensate size but for the moment of order 1 there is no simple relationship. 
It turns out that the first random model is completely dominated by Kingman's model in terms of condensate size and moments of all orders of the equilibrium. We conjecture that the first random model is also dominated by the second random model in the same sense, as supported by a different comparison in Theorem \ref{iidofa}. The relationship between Kingman's model and the second random model is more subtle. 

\section{Perspectives}

Recently, the phenomenon of condensation has been studied a lot in the literature. Biaconi {\it et al} \cite{BFF09} argued that the phase transition of condensation phenomenon is very close to Bose-Einstein condensation where a large fraction of a dilute gas of bosons cooled to temperatures very close to absolute zero occupy the lowest quantum state. See also \cite{BB11} for another model which can be mapped into the physics context. Under some assumptions, Dereich and M\"orters \cite{DM13} studied the limit of the scaled shape of the traveling wave of mass towards the condensation point in Kingman's model, and the limit turns out to be of the shape of some gamma function. A series of papers \cite{ D16, VDM18, DM15, MMU16, DMM17} were written later on to investigate the shape of traveling wave in other models where condensation appears and have proved that gamma distribution is universal. Park and Krug \cite{PK08} adapted Kingman's model to a finite population with unbounded fitness distribution and observed in a particular case emergence of Gaussian distribution as the wave travels to infinity. 

The first random model, as a natural random variant of Kingman's model, provides an interesting example to study condensation in detail. The matrix representation can be a handy tool to study the shape of the traveling wave to verify if the gamma-shape conjecture holds. On the other hand, we can also ask the question: will the relationships between the three models revealed and conjectured in this paper be applicable to other more sophisticated models under the competition of two forces, particularly to those models on the balance of selection and mutation? It is very tempting to say yes. The verification of the universality constitutes a long term project.

\section{Proofs}

\subsection{Proof of Lemma \ref{matrix}}

\begin{proof}[\bf Proof of Lemma \ref{matrix}]

Note that $$\frac{xP_{n}^n(dx)}{\int yP_{n}^n(dy)}=\frac{xQ(dx)}{m_1}=\frac{|W_x^{n+1,n}|}{|W^{n+1,n}|}Q(dx).$$
Assume that for some $0\leq j\leq n-1$, $$\frac{xP_{j+1}^n(dx)}{\int y P_{j+1}^n(dy)}=\frac{| W_x^{j+2,n}|}{|W^{j+2,n}|}Q(dx).$$
Then $$P_{j}^n(dx)=(1-b_{j+1})\frac{|W_x^{j+2,n}|}{|W^{j+2,n}|}Q(dx)+b_{j+1}Q(dx).$$ Consequently
\begin{align*}
\frac{xP_{j}^n(dx)}{\int yP_{j}^n(dy)}&=\frac{(1-b_{j+1})x\frac{|W_x^{j+2,n}|}{| W^{j+2,n}|}+b_{j+1}x}{(1-b_{j+1})\int y\frac{| W_y^{j+2,n}|}{| W^{j+2,n}|}Q(dy)+b_{j+1}m_1}Q(dx)&\\
&=\frac{\gamma_jx|W_x^{j+2,n} |+x|W^{j+2,n}|}{\gamma_j |U^rW^{j+2,n} |+m_1|W^{j+2,n}|}Q(dx)=\frac{|W_x^{j+1,n}|}{|W^{j+1,n}|}Q(dx).&
\end{align*}
The last equality is obtained by expanding $W_x^{j+1,n}$ and $W^{j+1,n}$ on the first column. By induction, we prove (\ref{p0n}). As a consequence, we also get (\ref{pphi}).% We thus conclude by induction. 
\end{proof}
Lemma \ref{matrix} allows us to express $P_j^n$ using $\{Q^j,Q^{j+1},\cdots,Q^{n-j}\}$. 
To write down the explicit expression, we introduce  $$\Phi_{j,l,n}:=\left(\prod_{i=0}^{l-1}\gamma_{i+j}L_{i+j,n}\right)L_{j+l,n}m_{l+1},\quad n\geq j\geq 1, l\geq 0.$$
\begin{corollary}\label{pjn} For $(P_j^n)$ with $P_n^n=Q$
\begin{align}\label{pmn}
P_{j}^n(dx)=\sum_{l=0}^{n-j}C^n_{j,l}Q^{l}(dx),\quad  \,\,0\leq j\leq n-1
\end{align}
where $C^n_{j,0}=b_{j+1};\quad  C^n_{j,l}=(1-b_{j+1})\Phi_{j+2,l-1,n},\quad   1\leq l\leq n-j.$

\end{corollary}
\begin{proof}Let $0\leq j\leq n-1$. Note that for any $1\leq l\leq n-j$
$$\frac{|W^{j+l,n}|}{|W^{j,n}|}=\prod_{i=0}^{l-1}\frac{|W^{i+j+1,n}|}{|W^{i+j,n}|}=\prod_{i=0}^{l-1}L_{i+j,n}.$$
Expanding the first row of $W_x^{j,n}$ and using the above result, we get 
\begin{align}\label{aanx}
\frac{|W_x^{j,n}|}{|W^{j,n}|}&=\frac{1}{|W^{j,n}|}\sum_{l=1}^{n-j+2}\left(\prod_{i=0}^{l-2}\gamma_{i+j}\right)|W^{j+l,n} |x^{l}&\nonumber\\
&=\sum_{l=1}^{n-j+2}\left(\prod_{i=0}^{l-2}\gamma_{i+j}L_{i+j,n}\right)L_{j+l-1,n}x^{l}=\sum_{l=1}^{n-j+2}\Phi_{j,l-1,n}\frac{x^{l}}{m_{l}}.&\end{align}
Then we plug it in (\ref{pphi}), changing $j$ to $j+2$. 
\end{proof}

\subsection{Proof of Lemma \ref{rl}}
We need to prove first a few more results on monotonicity. The following H\"older's inequality will be heavily used:
\begin{equation}\label{holder}\frac{m_{j+1}}{m_{j+2}}< \frac{m_j}{m_{j+1}}<\frac{1}{m_1}, \quad \forall j\geq 1.\end{equation}
\begin{lemma}\label{gk1}
For $j\geq 1, n\geq j-1$, $R^n_j$ increases strictly in $n$ to $R_j\in (0,1]$, as
$$m_1<R^n_{j}<R^{n+1}_{j}< 1 .$$
%Denote the limit of $R^n_{i+1,1}$ by $G_{i+1,1}$, and the limit of $R^n_{1}$ by $G_{1}$.
\end{lemma}
\begin{proof}
%Without additional claims, the determinants of matrices appearing below are all strictly positive. 
%Take $i=1.$ It suffices to show that the second term in (\ref{laa}) decreases in $n$. But instead we prove $R^n_{2,1}$ increases in $n$. 
%\begin{equation}\label{ura2n}R^n_{2,1}.\end{equation}
By H\"older's inequality, for $ j=n+1$,  
$$m_1<R^{n}_{n+1}=\frac{m_2}{m_1}<\frac{m_1m_2+\gamma_{n+1}m_3}{m_1^2+\gamma_{n+1}m_2}=R^{n+1}_{n+1}< 1.$$
Consider $n\geq j.$ Without loss of generality let $j=1.$ Using (\ref{lr})
$$R_{1}^n=\frac{|U^rW^{n}|}{|W^{n}|}.$$
The two matrices  $U^rW^{n}, W^{n}$ differ only on the first row, which is $(m_2, \cdots, m_{n+2})$ for the former, and $(m_1, \cdots, m_{n+1})$ for the latter. 
Again by H\"older's inequality, we have 
$$m_1< R^n_{1}<1,\,\,\forall \,n\geq 1.$$ %, for the purpose of notational simplicity
For the comparison of $R^n_{1}$ and $R^{n+1}_{1}$, we use Lemma \ref{xy} in the Appendix
where the values $x_0^n,x_0^{n+1}$ are exactly $R^n_{1}$ and $R^{n+1}_{1}$.

%Note that 
%$$\frac{|A^{n+1}|}{|U^rA^{n+1}|}=\frac{m_1\Big |A^{n} |+\gamma_{n+1} |U^cA^{n}|}{m_1 |U^rA^{n} |+\gamma_{n+1} |U^rU^cA^{n}|}.$$
%So it is equivalent to prove that for any $n\geq 1$
%$$\frac{|A^{n}|}{|U^rA^{n}|}> \frac{|U^cA^{n}|}{|U^rU^cA^{n}|}.$$
%or
%$$\frac{|A^{n}|}{|U^cA^{n}|}\geq \frac{|U^rA^{n}|}{|U^rU^cA^{n}|}.$$
%One can easily check that the monotonicity holds for at least $n=1,2,3.$
%Recall $X^n,Y^n$ from Lemma \ref{xy},
%$$x_0^n=\frac{|U^rW^n|}{|W^n|}, y_0^n=\frac{|U^rU^cW^n|}{|U^cW^n|}.$$
%So one has to show 
%$$x_0^n< y_0^n.$$
%By Lemma \ref{xy} below, $X$ and $Y$ are strictly positive. 
%Notice that $$X^nC_{n+1}(A^{n})=x_0^nm_{n+1}+x_1^nm_{n}+\cdots+x_{n}^nm_1=m_{n+2}.$$
%So
%$$X^nC_{n+1}(U^cW^n)=x_0^nm_{n+2}+x_1^nm_{n+1}+\cdots+x_{n}^nm_2< m_{n+3}.$$
%Then 
%$$X^nW^n-X^nU^cW^n=(0,0,\cdots,0,a)$$
%with some $a>m_{n+2}-m_{n+3}$. But 
%$$X^nW^n-Y^nU^cW^n=m_{n+2}-m_{n+3}.$$
%Proceeding similarly as in the proof of Lemma \ref{xy}, we can construct $Y^n$ from $X^n$ and obtain that
%$0<x_0^n< y_0^n.$%,\text{ if }\gamma_i\neq \infty; \quad x_i^n=y_i^n,\text{ otherwise}.$$
 %The bounds of $L_i$ comes from (\ref{lbd}). 
\end{proof}
Simply applying the above lemma and (\ref{laa}), we obtain the following Corollary.  
\begin{corollary}\label{lphi}
For any $j\geq 1$, $\gamma_j L_{j,n}$ decreases strictly in $n$ to $\gamma_j L_{j}$. Define $$\Phi_{j,l}:=\left(\prod_{i=0}^{l-1}\gamma_{i+j}L_{i+j}\right)L_{j+l}m_{l+1},\quad \forall j\geq 1, l\geq 0.$$ Then $\Phi_{j,l,n}=\Phi_{j,l}=0$ if $\gamma_{j+l}=\infty$, otherwise $\Phi_{j,l,n}$ decreases strictly  in $n$ to $\Phi_{j,l}$. 
\end{corollary}

\begin{corollary}\label{ukn}
For any $j\geq 1, l\geq 1,$ $R^n_{j,k}$ increases strictly in $n$ to $R_{j,k}$.
\end{corollary}
\begin{proof}
The case $k=1$ has been proved by Lemma \ref{gk1}. We consider here $k\geq 2$. Without loss of generality we let $j=1.$ The idea is to apply Lemma \ref{simplelemma} in the Appendix. Following the notations in Lemma \ref{simplelemma} we set 
$$a_l=\int y^{k+1}Q^l(dy)=\frac{m_{l+k+1}}{m_l},\quad b_l=\int yQ^l(dy)=\frac{m_{l+1}}{m_l}, \quad\forall \,0\leq l\leq n; $$
and 
$$c_l=C^{n-1}_{0,l},\,\,\, c'_l=C^n_{0,l}, \,\,\forall \,0\leq l\leq n-1; \quad c_n=0,\,\,\, c'_n=C^n_{0,n}.$$
Then by the definition of $R_{1,k}^n$ and Lemma \ref{matrix}
\begin{equation}\label{rp}R^{n-1}_{1,k}=\frac{|U_k^rW^{n-1}|}{|W^{n-1}|}=\frac{\int y^{k+1} P_{0}^{n-1}(dy)}{\int y P_{0}^{n-1}(dy)}.\end{equation}
So by (\ref{pmn})$$R^{n-1}_{1,k}=\frac{\sum_{l=0}^{n}c_la_l}{\sum_{l=0}^{n}c_lb_l},\quad R^{n}_{1,k}=\frac{\sum_{l=0}^{n}c'_la_l}{\sum_{l=0}^{n}c'_lb_l}.$$
For any  $n\geq 1$, by H\"older's inequality
$$\frac{a_l}{b_l}=\frac{m_{l+k+1}}{m_{l+1}}<\frac{ m_{n+k+1}}{m_{n+1}}=\frac{a_{n}}{b_{n}}, $$
and $$a_l=\frac{m_{l+k+1}}{m_{l}}<\frac{ m_{n+k+1}}{m_{n}}=a_{n}, \quad b_l=\frac{m_{l+k+1}}{m_{l}}<\frac{ m_{n+k+1}}{m_{n}}=b_{n},\quad \forall \,0\leq l\leq n-1. $$
Moreover $a_0,\cdots, a_n, b_0,\cdots, b_n$ are all strictly positive numbers.

Next we consider the $c_l$'s and $c_l'$'s.  Note that $c_0=c_0'=b_1$. By Corollary \ref{lphi}, for $1\leq l\leq n-1,$ if $c_l>0$, then $c_l>c_l'$, otherwise $c_l=c_l'=0.$ Moreover $c_n'=C_{0,n}^n=(1-b_1)\frac{m_{n+1}}{m_1}\prod_{i=0}^{n-1}\gamma_iL_{i,n}>0$. So we have the following 
$$c_i\geq c'_i\geq 0, \quad \forall \,0\leq l\leq n-1;\quad \,\,0=c_n<c'_n;\quad \,\,\sum_{i=1}^nc_i=\sum_{i=1}^nc'_i=1.$$
Now we apply Lemma \ref{simplelemma} to conclude. 
%,  
%This result generalizes the strict monotonicity in the case $k=1$ shown in Remark \ref{llimit}. 
\end{proof}
\begin{proof}[\bf Proof of Lemma \ref{rl}.]
As we have already proved Corollary \ref{lphi} and \ref{ukn}, it remains to tackle (\ref{rik}) and (\ref{boundrjlj}). 
Expanding $U_k^rW^{j,n}$ and $W^{j,n}$ on the first column, we get
$$R^n_{j,k}=\frac{|U_k^rW^{j,n}|}{|W^{j,n}|}=\frac{m_{k+1}|W^{j+1,n} |+\gamma_j  |U_{k+1}^rW^{j+1,n}|}{m_{k+1}|W^{j+1,n} |+\gamma_j  |U^rW^{j+1,n}|}=\frac{m_{k+1}+\gamma_jR^n_{j+1,k+1}}{m_1+\gamma_jR^n_{j+1}}.$$
Letting $n\to\infty$, we obtain (\ref{rik}). 

To show (\ref{boundrjlj}), without loss of generality, let $j=1$.   By Lemma \ref{gk1}
$$m_1<R_{2,1}^n<1.$$
As $R_{2,1}^n$ decreases to $R_{2,1}$, we have also $R_{2,1}<1$ which gives the strict upper bound for $R_{2,1}$. Using (\ref{laa}), the above display yields
\begin{equation}\label{lbd}
\frac{\gamma_1}{m_1+\gamma_1}< \gamma_1L_{1,n}< \frac{\gamma_1}{m_1(1+\gamma_1)}.
\end{equation}
Since $\gamma_1L_{1,n}$ decreases strictly to $\gamma_1L_1$, we obtain the following using again (\ref{laa})
$$\gamma_1L_1=\frac{\gamma_1}{m_1+\gamma_1R_{2,1}}<\frac{\gamma_1}{m_1(1+\gamma_1)}.$$
Then we get $R_{2,1}>m_1$. Moreover as $R_{2,1}<1,$
$$\gamma_1L_1=\frac{\gamma_1}{m_1+\gamma_1R_{2,1}}>\frac{\gamma_1}{m_1+\gamma_1}.$$
So we have found the strict lower and upper bounds for $R_{2,1}$ and $\gamma_1L_1$.

\end{proof}

\subsection{Proofs of Theorem \ref{h} and Corollary \ref{h=g}}
For measures $u,v\in M_1$, we write 
$$u\leq v$$
if $u([0,x])\geq v([0,x])$ for any $x\in[0,1]$. 
\begin{proof}[\bf Proof of Theorem \ref{h}]
Note that $Q^j\leq  Q^{j+1}$ for any $j$. Then using Corollary \ref{pjn} and Lemma \ref{rl}, $P_j^n\leq P_j^{n+1}$. So $P_j^n$ converges at least weakly to a limit $\H_j$. The weak convergence allows to obtain (\ref{hj+1}) from (\ref{revp}). Expanding  (\ref{hj+1}), we obtain
\begin{align}\label{hjhj}
\H_j(dx)&=H_j\delta_{S_Q}(dx)+b_{j+1}Q(dx)+\sum_{l=1}^{\infty}(1-b_{j+1})\Phi_{j+2,l-1}Q^{l}(dx),\,  \,0\leq j< n.&
\end{align}
where $H_j=1-b_{j+1}-\sum_{l=1}^{\infty}(1-b_{j+1})\Phi_{j+2,l-1}.$
 To prove (\ref{hl}), we firstly use 
(\ref{pphi}) and definition (\ref{lr}) to obtain that 
\begin{align*}\int xP_j^n(dx)&=(1-b_{j+1})\frac{|U^rW^{j+2,n}|}{|W^{j+2,n}|}
+b_{j+1}m_1&\\
&=(1-b_{j+1})R_{j+2}^n+b_{j+1}m_1=b_{j+1}(\gamma_{j+1}R_{j+2}^n+m_1)=\frac{b_{j+1}}{L_{j+1,n}}.&\end{align*}
A reformulation of the above equality reads 
$$\frac{1-b_{j+1}}{\int yP_j^n(dy)}=\gamma_jL_{j+1,n}.$$
Using the convergences as $n\to\infty$, we obtain (\ref{hl}). 
\end{proof}
\begin{proof}[\bf Proof of Corollary \ref{h=g}]
By (\ref{hjh}), $\widetilde \H_j$ is equal in distribution for all $j$'s. By (\ref{hj+1}), $\widetilde \H_j$ is an invariant measure on $[0,S_Q]$ with $S_{\widetilde \H_j}=S_Q$ a.s.. Recall that $\I_{j,S_Q}$ is also invariant on $[0,S_Q]$. Then by Theorem 4 in \cite{Y20}, $\widetilde \H_j\stackrel{d}{=}\I_{j,S_Q}.$  By (\ref{gf}) and (\ref{hj+1}), for both sequences, the multi-dimensional distributions are determined in the same way by one dimensional distribution. So the two sequences have the same multi-dimensional distributions, and the multi-dimensional distributions are consistent in each sequence.  By Kolmogorov's extension theorem (Theorem 5.16, \cite{K97}), consistent multi-dimensional distributions determine the distribution of the sequence, which yields the identical distribution for both two sequences. 
\end{proof}
%The convergence is strong if and only if 
%$$b_{j+1}+\sum_{l=1}^{\infty}(1-b_{j+1})\Phi_{j+2,l-1}=1.$$

\subsection{Proof of Corollary \ref{con}}

\begin{proof}[\bf Proof of Corollary \ref{con}]
Recall that 
$\E\left[\frac{1-\beta}{\int y\I_Q} \right]$ exists and does not depend on the joint law of 
$\beta, \I_Q$. Using (\ref{hl}) in the first random model, together with Corollary \ref{h=g}, we can rewrite Theorem \ref{conY20} into Corollary \ref{con}.  

\end{proof}

\subsection{Proof of Lemma \ref{key}}

\begin{proof}[\bf Proof of Lemma \ref{key}]
Since $(\xi_1,\cdots,\xi_n)$ is exchangeable, we can directly take a symmetric function $f$ and prove the inequality under  $f_{x_1x_2}\leq 0.$
For any $a>b$, we first show that
$$f(a,\underbrace{b,\cdots,b}_{n-1})+f(b,\underbrace{a,\cdots,a}_{n-1})\geq f(\underbrace{a,\cdots,a}_{n})+f(\underbrace{b,\cdots,b}_{n}),$$
which is proved as follows. 
\begin{align*}
&f(\underbrace{a,\cdots,a}_{n})+f(\underbrace{b,\cdots,b}_{n})-f(a,\underbrace{b,\cdots,b}_{n-1})-f(b,\underbrace{a,\cdots,a}_{n-1})&\\
=&\int_b^a(f_{x_1}(x_1, \underbrace{a,\cdots,a}_{n-1})-f_{x_1}(x_1, \underbrace{b,\cdots,b}_{n-1}))dx_1&\\
=&\sum_{i=2}^n\int_b^a(f_{x_1}(x_1,\underbrace{b,\cdots,b}_{i-2},a,\underbrace{a,\cdots,a}_{n-i})-f_{x_1}(x_1,\underbrace{b,\cdots,b}_{i-2},b,\underbrace{a,\cdots,a}_{n-i}))dx_1&\\
=&\sum_{i=2}^n\int_b^a\int_b^af_{x_1x_i}(x_1,\underbrace{b,\cdots,b}_{i-2},x_i,\underbrace{a,\cdots,a}_{n-i})dx_1d{x_i}&\\
=&\sum_{i=2}^n\int_b^a\int_b^af_{x_1x_2}(x_1,x_2,\underbrace{b,\cdots,b}_{i-2},\underbrace{a,\cdots,a}_{n-i})dx_1d{x_2}\leq 0&
\end{align*}
Applying the proved result, for any $1\leq i\leq n-1$, 
\begin{align*}&f(\underbrace{\xi_1,\cdots,\xi_1}_i,\xi_{i+1},\xi_{i+2},\cdots,\xi_n)+f(\underbrace{\xi_{i+1},\cdots,\xi_{i+1}}_i,\xi_{1},\xi_{i+2},\cdots,\xi_n)&\\
\geq& f(\underbrace{\xi_1,\cdots,\xi_1}_{i+1},\xi_{i+2},\cdots,\xi_n)+f(\underbrace{\xi_{i+1},\cdots,\xi_{i+1}}_{i+1},\xi_{i+2},\cdots,\xi_n).&\end{align*}
Using the above inequality, we obtain\begin{align*}
&\E[f(\underbrace{\xi_1,\cdots,\xi_1}_i,\xi_{i+1},\xi_{i+2},\cdots,\xi_n)]&\\
=&\frac{1}{2}\E[f(\underbrace{\xi_1,\cdots,\xi_1}_i,\xi_{i+1},\xi_{i+2},\cdots,\xi_n)+f(\underbrace{\xi_{i+1},\cdots,\xi_{i+1}}_i,\xi_{1},\xi_{i+2},\cdots,\xi_n)]&\\
\geq& \frac{1}{2}\E[f(\underbrace{\xi_1,\cdots,\xi_1}_{i+1},\xi_{i+2},\cdots,\xi_n)+f(\underbrace{\xi_{i+1},\cdots,\xi_{i+1}}_{i+1},\xi_{i+2},\cdots,\xi_n)]&\\
=&\E[f(\underbrace{\xi_1,\cdots,\xi_1}_{i+1},\xi_{i+2},\cdots,\xi_n)].&
\end{align*}
Letting $i$ travel from $1$ to $n-1$,  we prove the lemma. 

\end{proof}

\subsection{Proof of Theorem \ref{gammal3}}

Define
$$\Psi_n:=\frac{\prod_{j=1}^{n}\gamma_j}{|W^n|}, \quad n\geq 1.$$
\begin{lemma}\label{rn}
%\begin{itemize}
%\item If $\E[\ln \gamma_1L_1]\geq  0$, then $\E[\ln R_n]$ increases and 
%\item If $\E[\ln \gamma_1L_1]< 0$, then $\E[\ln R_n]$ converges to $-\infty$.
%\end{itemize}
For the three models, we have 
$$\lim_{n\to\infty }\frac{\ln \overline \Psi_n}{n}=\ln \gamma \overline L, \quad \lim_{n\to\infty }\frac{\ln \widehat \Psi_n}{n}=\ln \Gamma \widehat L,\quad \lim_{n\to\infty }\E\left[\frac{\ln \widetilde \Psi_n}{n}\right]=\E\left[\ln \Gamma_1 \widetilde L_{1}\right].$$  \end{lemma}
\begin{proof}
We prove only the case in the first random model. Note that 
\begin{align*}\E[\ln \widetilde \Psi_n]&=\E\left[\ln\left(\frac{1}{m_1}\prod_{j=1}^{n-1} \Gamma_j \frac{|\widetilde  W^{j+1,n}|}{|\widetilde  W^{j,n}|}\right)\right]&\\
&=\sum_{j=1}^{n-1}\E[\ln(\Gamma_j \widetilde L_{j,n})]-\ln m_1=\sum_{j=1}^{n-1}\E[\ln(\Gamma_1 \widetilde L_{1,n-j+1})]-\ln m_1.&\end{align*}
Here we use the fact that $\Gamma_j \widetilde L_{j,n}\stackrel{d}{=}\Gamma_1 \widetilde L_{1,n-j+1}.$
Then we apply Lemma \ref{rl}.  
\end{proof}

\begin{lemma}\label{rconv}

% $1). \ln R_n$ is strictly increasing and strictly concave in each $\gamma_i, 1\leq i\leq n-1$. 
%\\
$\ln \Psi_n$ is strictly concave down in every $b_j, 1\leq j\leq n$. 

%In consequence, letting $B_\gamma^n$ (resp. $B_{\beta}^n$) be $R_n$ with each $\gamma_i$ (resp. $\beta_i$) replaced by $E[\gamma_1]$ (resp. $\E[\beta_1]$), we have
%$$\E[\ln R_n]<\min\{\E[\ln B_\betW^n],\E[\ln B_\gamma^n]\}=\E[\ln B_\betW^n]$$
%which entails 
%$$\lim_{n\to\infty}\E[\ln R_n]\leq \lim_{n\to\infty}\min\{\E[\ln B_\betW^n],\E[\ln B_\gamma^n]\}=\lim_{n\to\infty}\E[\ln B_\betW^n].$$
\end{lemma}
\begin{proof}
By basic computations we obtain for $b_j\in(0,1)$, 
%$$\frac{\partial \ln M_n}{\partial b_j}=-\left(1/\gamma_j-\frac{d|W^n|}{d\gamma_j}/|W^n|\right)/b_j^2<0$$
%and 
$$\frac{\partial^2 \ln \Psi_n}{\partial b_j^2}=\frac{1}{b_j^4}\left(1/\gamma_j-\frac{d|W^n|}{d\gamma_j}/|W^n|\right)\left(2b_j-1/\gamma_j-\frac{d|W^n|}{d\gamma_j}/|W^n|\right).$$
By Lemma \ref{aga} in the Appendix, $\frac{\partial^2 \ln \Psi_n}{\partial b_j^2}<0.$
\end{proof}

\begin{proof}[\bf Proof of Theorem \ref{gammal3}]
To prove (\ref{3}), we can use Lemma \ref{rn} and show instead 
\begin{equation}\label{3m}\E[\ln \widehat \Psi_n]\leq \E[\ln \widetilde \Psi_n]\leq \ln \overline \Psi_n.\end{equation} 
For any $1\leq j<i\leq n$, due to Proposition \ref{wij} in the Appendix, $$\frac{\partial^2\ln \Psi_n}{\partial b_i\partial b_j}=-\frac{\partial^2\ln |W^n|}{\partial b_i\partial b_j}<0.$$
Then we apply Lemma \ref{key} to obtain the first inequality of (\ref{3m}). Next we apply Lemma \ref{rconv} and Janson's inequality for the second inequality of (\ref{3m}). To prove (\ref{bargammal}), we use (\ref{hl}), and Theorem \ref{King}.
\end{proof}

\subsection{Proof of Theorem \ref{randet}}
We need two preparatory results before proving the theorem. 
\begin{lemma}\label{mono}
For any $k,n$, $R^n_{1,k}$ is strictly concave down in every $b_i$, $1\leq i\leq n$. %The same convexity holds for $R_k$.
%In consequence, $H_{0}$ is increasing (resp. decreasing) and concave down in every $\gamma_j$ (resp. $b_j$), $j\geq 1$.
\end{lemma}
\begin{proof}%We shall consider first $\gamma_i$. 
Let $b_i\in(0,1)$. Let $$f= |U_k^rW^n |, \, g=|W^n|.$$ 
So $R^n_{1,k}=\frac{f}{g}$. Let $f',f'',g',g''$ be derivatives with respect to $\gamma_i\in (0,\infty)$. 
Then by Corollary \ref{uk} in the Appendix
$$\frac{dR^n_{1,k}}{d\gamma_i}=\frac{f'g-fg'}{g^2}>0$$
%where the inequality is due to Corollary \ref{uk}. Note that
Notice  that $$\frac{g'}{g}>0,\quad \frac{f''}{g}=\frac{g''}{g}=0.$$
The above statements are not difficult to see if it is clear how $f,g$ can be computed. Or one can refer to Lemma \ref{type*} in the Appendix. 
Then we obtain 
$$\frac{d^2R^n_{1,k}}{d(\gamma_i)^2}=\frac{f''g-fg''}{g^2}-\frac{2g'}{g}\frac{f'g-fg'}{g^2}=-\frac{2g'}{g}\frac{dR^n_{1,k}}{d\gamma_i}<0.$$
%Thus the part for $\gamma_i$ is proved. 
Moreover, $$\frac{d\gamma_i}{db_i}=\frac{-1}{b_i^2}, \quad \frac{d^2\gamma_i}{d(b_i)^2}=\frac{2}{b_i^3}.$$
Then
%$$dR^n_{k}/d\gamma_i=\frac{f'g-fg'}{g^2}\frac{-1}{b_i^2}<0$$
%and 
\begin{align*}
\frac{d^2R^n_{1,k}}{d(b_i)^2}&=\left(\frac{-1}{b_i^2}\right)^2\frac{d^2R^n_{k}}{d(\gamma_i)^2}+\frac{2}{b_i^3}\frac{dR^n_{1,k}}{d\gamma_i}&\\
&=\frac{2(f'g-fg')}{g^2b_i^4}\left(b_i-\frac{g'}{g}\right)=\frac{2}{b_i^4}\frac{dR^n_{1,k}}{d\gamma_i}\left(b_i-\frac{g'}{g}\right)<0,\end{align*}
where the inequality is due to Lemma \ref{aga} in the Appendix. %The strict convexity of $R_{k}$ follows by letting $n$ go to infinity and using again Corollary \ref{uk}, Lemma \ref{aga}. %The part for $b_i$ is proved. 

%For the conclusion on $C_{0,*}$, we refer to Remark \ref{ci*}. 
\end{proof}

\begin{corollary}\label{hjj+1}
For $H_j$ defined in (\ref{hjhj}), we have 
\begin{equation}\label{hjb}\frac{H_j}{1-b_{j+1}}=S_Q\gamma_{j+2}L_{j+2}\frac{H_{j+1}}{1-b_{j+2}},\end{equation}
and if $Q(S_Q)=0$,
\begin{equation}\label{hj1-b}\frac{H_j}{1-b_{j+1}}=\lim_{k\to\infty}S_Q^{-k}R_{j+2,k}.\end{equation}
\end{corollary}
\begin{proof}
By (\ref{hj+1}), we obtain
$$H_j=\frac{1-b_{j+1}}{\int y \H_{j+1}(dy)}S_QH_{j+1}.$$
The above display together with (\ref{hl}) lead to (\ref{hjb}). If $Q(S_Q)=0$, then $\lim_{k\to\infty}S_Q^{-k}m_{k+1}=0$. Using this fact and (\ref{pphi}), we obtain 
\begin{align*}
H_j&=\H_j(S_Q)=\lim_{k\to\infty}S_Q^{-k}\int y^k\H_j(dy)&\\
&=\lim_{k\to\infty}\lim_{n\to\infty}S_Q^{-k}\int y^kP_j^n(dy)&\\
&=\lim_{k\to\infty}\lim_{n\to\infty}S_Q^{-k}\left((1-b_{j+1})R_{j+2,k}^n+b_{j+1}m_{k+1} \right)&\\
&=(1-b_{j+1})\lim_{k\to\infty}\lim_{n\to\infty}S_Q^{-k}R_{j+2,k}^n=(1-b_{j+1})\lim_{k\to\infty}S_Q^{-k}R_{j+2,k}.&
\end{align*}
\end{proof}

\begin{proof}[{\bf Proof of Theorem \ref{randet}}]
There are two statements to prove. 

\vspace{3 mm}

\noindent 1. 
By (\ref{rik})
$$R_{1,k}=\frac{m_{k+1}+\gamma_1R_{2,k+1}}{m_1+\gamma_1R_{2}}.$$
By Corollary \ref{uk} in the Appendix, $R_{1,k}$ is strictly increasing in $\gamma_1.$ Then 
$$R_{1,k}>\frac{m_{k+1}}{m_1}$$
implying that 
$$\frac{m_{k+1}}{R_{2,k+1}}<\frac{m_{1}}{R_{2}}.$$
The above inequality entails that for $b_1\in(0,1)$
$$\frac{\partial^2 R_{1,k} }{\partial b_1^2}=\frac{2}{(1+\frac{m_1}{R_2}-b_1)^3}(1+\frac{m_1}{R_2})\frac{R_{2,k+1}}{R_2}(\frac{m_{k+1}}{R_{2,k+1}}-\frac{m_{1}}{R_{2}})<0.$$
So $R_{1,k}$ is strictly concave down in $b_1.$

In the following display, the first equality is due to (\ref{pphi}) and the first inequality is by the above strict concavity. The second equality is due to Lemma \ref{rl} and the second inequality is by Lemma \ref{mono}. The last equality is a consequence of (\ref{pphi}) and Corollary \ref{ukn}.  \begin{align*}
\E\left[\int y^k\I_{S_Q}(dy)\right]&=(1-b)\E[\widetilde R_{1,k}]+bm_k&\\
&<(1-b)\E[\widetilde R_{1,k} |{\beta_1=b}]+bm_k&\\
&=(1-b)\lim_{n\to\infty}\E[\widetilde R^n_{1,k}|{\beta_1=b}]+bm_k&\\
&\leq (1-b)\lim_{n\to\infty}\overline R^n_{1,k}+bm_k&\\
&=\int y^k\K_{Q}(dy).&
\end{align*}

\noindent 2.  By Corollary \ref{h=g}, $I_Q\stackrel{d}{=}\widetilde H_{0}$. Since $I_Q>0$ a.s., by assertion 4) of Corollary 4 in \cite{Y20}, we have $Q(S_Q)=0$. Note that 
$\widetilde H_j/(1-\beta_{j+1})$ involves only $\beta_{j+2},\beta_{j+3},\cdots.$ Then by (\ref{hjb}), 
\begin{align*}\E[I_Q]=\E[\widetilde H_0]&=\E\left[(1-\beta_1)\frac{\widetilde H_0}{1-\beta_1}\right]&\\
&=(1-b)\E\left[\frac{\widetilde H_0}{1-\beta_1}\right]=(1-b)S_Q\E\left[\Gamma_{2}\widetilde L_2\frac{\widetilde H_1}{1-\beta_{2}}\right].&\end{align*}
Moreover for $b_2\in (0,1)$
$$\gamma_2L_2=\frac{1-b_2}{b_2m_1+(1-b_2)R_{3,1}}$$
and by (\ref{boundrjlj})
$$\frac{\partial^2\gamma_2L_2}{\partial b_2^2}=\frac{2m_1(m_1-R_{3,1})}{(b_2m_1+(1-b_2)R_{3,1})^3}<0.$$
So the function $\gamma_{2}L_2\frac{H_1}{1-b_{2}}$ is strictly concave down on $b_2$, as $\frac{H_1}{1-b_{2}}$ does not depend on $b_2$. 
Using (\ref{hj1-b}) and the above strict concavity, together with Lemma \ref{mono}, 
\begin{align*}
\E[\widetilde H_0]&=(1-b)\E\left[\frac{\widetilde H_0}{1-\beta_1}\right]<(1-b)\E\left[\frac{\widetilde H_0}{1-\beta_1} \Big |\beta_2=b\right]&\\
&=(1-b)\lim_{k\to\infty}S_Q^{-k}\E[\widetilde R_{2,k}| {\beta_2=b}]&\\
&= (1-b)\lim_{k\to\infty}\lim_{n\to\infty}S_Q^{-k}\E[\widetilde R^n_{2,k}|{\beta_2=b}]&\\
&\leq (1-b)\lim_{k\to\infty}\lim_{n\to\infty}S_Q^{-k}\E[\widetilde R^n_{2,k}|{\beta_i=b, \forall i\geq 2}]=(1-b)\frac{\overline H_0}{1-b}=\overline H_0.&
\end{align*}

\end{proof}

\subsection{Proof of Theorem \ref{iidofa}}
\begin{proof}[\bf Proof of Theorem \ref{iidofa}]
Note that similarly as in the proof of Lemma \ref{rn}
\begin{align*}
\E\left[\ln \left(|\widetilde W^n|\prod_{j=1}^n\beta_j\right)\right] &=\E\left[\ln\left(\frac{1}{m_1}\prod_{j=1}^{n-1} \beta_j \frac{|\widetilde W^{j,n}|}{| \widetilde W^{j+1,n} |}\right)\right]=\sum_{j=1}^{n-1}\E\left[\ln\frac{\beta_1}{\widetilde L_{1,n-j+1}}\right]-\ln m_1.&\\
\end{align*}
For the second random model, similarly
\begin{align*}
\E\left[\ln \left(|\widehat W^n|\beta^n\right)\right]=\sum_{j=1}^{n-1}\E\left[\ln\frac{\beta}{\widehat L_{1,n-j+1}}\right]-\ln m_1.&
\end{align*}
By Lemma \ref{rik} and (\ref{hl}),  
\begin{equation}\label{betaj}\lim_{n\to\infty}\E\left[\ln \left(|\widetilde W^n|\prod_{j=1}^n\beta_j\right)\right] /n=\E\left[\ln \frac{\beta_1}{ \widetilde L_1}\right]=\E\left[\ln \int y \I_Q(dy)\right],\end{equation}
and 
\begin{equation}\label{betan}\lim_{n\to\infty} \E\left[\ln \left(|\widehat W^n|\beta^n\right)\right]/n=\E\left[\ln \frac{\beta}{ \widehat L}\right]=\E\left[\ln \int y \A_Q(dy)\right].\end{equation}
We compare  next $\E\left[\ln \left(|\widetilde W^n|\prod_{j=1}^n\beta_j\right)\right]$ and $\E\left[\ln \left(|\widehat W^n|\beta^n\right)\right].$ Note that 
$$\ln \left(| W^n|\prod_{j=1}^nb_i\right)=\ln |W^n|+\sum_{j=1}^n\ln b_j.$$
Then second order partial derivative of $\ln\left(| W^n|\prod_{j=1}^nb_i\right)$ with respect to $b_s,b_t$ equals $\frac{\partial^2\ln |W^n|}{\partial b_s\partial b_t}$ which is, by Lemma \ref{aga} in the Appendix, strictly positive for any $1\leq s\neq t\leq n$. Applying Lemma \ref{key}, we obtain 
$$\E\left[\ln \left(|\widetilde W^n|\prod_{j=1}^n\beta_j\right)\right] \leq\E\left[\ln \left(|\widehat W^n|\beta^n\right)\right]. $$
Then by (\ref{betaj}) and (\ref{betan}) we conclude that 
$$\E\left[\ln \int y \I_Q(dy)\right]\leq \E\left[\ln \int y \A_Q(dy)\right].$$
\end{proof}

\subsection{Proof of Theorem \ref{kinofa}}
\begin{proof}[\bf Proof of Theorem \ref{kinofa}]

By Theorem \ref{King},
\begin{equation*}  
K_Q=\left\{  
             \begin{array}{lr}  
            1-\int \frac{bQ(dx)}{1-x/S_Q}, &  \mbox{if $\frac{Q(dx)}{1-x/S_Q}<  b^{-1}$};\\  
             & \\  
             0, &    \mbox{if $\frac{Q(dx)}{1-x/S_Q}\geq  b^{-1}$}.
             \end{array}  
\right.  
\end{equation*}  
So $K_Q$ is a concave up function of $b$, and consequently $\E[A_Q]\geq K_Q.$

To show that there is no one-way inequality between $\E[\int y\A_Q(dy)]$ and $ \int y\K_Q(dy)$, we give a concrete example. Let $Q(dx)=dx$. In this case, $\int \frac{Q(dx)}{1-x/S_Q}=\int \frac{Q(dx)}{1-x}=\infty>  b^{-1}$ for any $b\in (0,1).$ By (\ref{bargammal})
$$\int y\K_Q(dy)=\theta_b$$
which satisfies equation $$\int\frac{b \theta_bdx}{\theta_b-(1-b)x}=1.$$ We show that $\frac{d^2\theta_b}{db^2}$ can be strictly positive and negative for different $b's.$ The above equation can be rewritten as 
$$\int\frac{bdx}{1-tx}=1$$
with $t=\frac{1-b}{\theta_b}\in (0,1)$ strictly decreasing in $b$. Then 
$$b=-\frac{t}{\ln(1-t)},\quad \theta_b=\frac{1}{t}+\frac{1}{\ln(1-t)}.$$
So
$$\frac{d\theta_b}{db}=\frac{d\theta_b/dt}{db/dt}=\frac{-(1-t)\ln^2(1-t)+t^2}{-(1-t)t^2\ln(1-t)-t^3}=\frac{m(t)}{n(t)}$$
with $m(t)$ the numerator and $n(t)$ the denominator. 
Then 
$$\frac{d^2\theta_b}{db^2}=\frac{d(d\theta_b/db)}{dt}/\frac{db}{dt}=\frac{m'(t)n(t)-m(t)n'(t)}{n(t)^2\frac{db}{dt}}$$
where 
\begin{align*}
&m'(t)n(t)-m(t)n'(t)&\\
&=-2t(1-t)^2\ln^3(1-t)+(-4t^2+3t^3)\ln^2(1-t)-t^3(2+t)\ln(1-t)&\\
&=5t^6+O(t^7), \quad t\to 0.&
\end{align*}
As $n(t)^2>0$ and $\frac{db}{dt}<0$ for any $t\in(0,1)$, we have $\frac{d^2\theta_b^2}{db^2}>0$ for $t$ small enough. However $m'(0.5)n(0.5)-m(0.5)n'(0.5)= -4.184810^{-4}<0$, implying $\frac{d^2\theta_b^2}{db^2}<0$ at $t=0.5$. As $t$ is a strictly decreasing function of $b$, we have shown that $\frac{d^2\theta_b^2}{db^2}$ can be strictly positive and negative at different $b$'s. 
\end{proof}

\section{Appendix}
\subsection{Appendix A}
\begin{lemma}\label{simplelemma}
Let $n>1.$ Let $a_0,\cdots, a_n, b_0,\cdots, b_n$ all be strictly positive numbers such that 
$$\frac{a_l}{b_l}<\frac{a_{n}}{b_{n}},\,\,\,a_l<a_{n},\,\,\, b_l<b_{n},\,\,\, \forall \,0\leq l\leq n-1. $$
Let $c_0,\cdots, c_n,  c'_0,\cdots, c'_n$ be nonnegative numbers such that 
$$c_l\geq c'_l, \,\,\, \forall \,0\leq l\leq n-1;\quad \,\,c_n<c'_n;\quad \,\,\sum_{l=1}^nc_l=\sum_{l=1}^nc'_l>0.$$
Then 
\begin{equation}\label{cc'}\frac{\sum_{l=1}^n{c_la_l}}{\sum_{i=1}^n{c_lb_l}}<\frac{\sum_{l=1}^n{c'_la_l}}{\sum_{l=1}^n{c'_lb_l}}.\end{equation}
\end{lemma}
\begin{proof}
Without loss of generality, assume $\sum_{l=1}^nc_l=1.$ Define 
$$A=\sum_{l=1}^n{c_la_l}=\sum_{l=1}^{n-1}{c_la_l}+\left(1-\sum_{l=1}^{n-1}c_l\right)a_n,  \quad B=\sum_{l=1}^{n-1}{c_lb_l}+\left(1-\sum_{l=1}^{n-1}c_l\right)b_n.$$
and 
$$f(c_0,\cdots,c_{n-1})=\frac{A}{B},\,\,\,\text{ with } c_l\geq 0,\,\,\, \sum_{l=0}^{n-1}c_l\in[0,1]. $$
To prove (\ref{cc'}), it suffices to show that for any $0\leq l\leq n-1$
$$\frac{\partial f}{\partial c_l}<0,\quad \forall\, c_l\in(0,1).$$
Without loss of generality, we consider only $l=0.$ We have
$$\frac{\partial f}{\partial c_0}=\frac{(b_n-b_0)A-(a_n-a_0)B}{B^2}.$$
Note that by the assumptions on $a_l$'s and $b_l$'s, 
$$\frac{a_n-a_0}{b_n-b_0}>\frac{a_n}{b_n}>\frac{a_l}{b_l}, \quad \forall\, 0\leq l\leq n-1.$$
That implies $$(b_n-b_0)A<(a_n-a_0)B$$
which entails $\frac{\partial f}{\partial c_0}<0.$
\end{proof}
%By Theorem \ref{main}, $L_1$ is of particular importance. In matrix representation, we have:
%\begin{equation}\label{rr}
%L_1=\lim_{n\to\infty}\frac{\left |\begin{array}{ccccc}m_{1}&m_{2}&m_{3}&\cdots&m_n\\
%-\zeta'&m_1&m_2&\cdots&m_{n-1}\\
%0&-\gamma_{3}&m_1&\cdots&m_{n-2}\\
%0&0&\ddots&\ddots&\vdots\\
%0&0&\ddots&-\gamma_{n}&m_1\\
%\end{array}\right |}{\left |\begin{array}{ccccc}m_{1}&m_{2}&m_{3}&\cdots&m_{n+1}\\
%-\gamma_1&m_1&m_2&\cdots&m_{n}\\
%0&-\gamma_{2}&m_1&\cdots&m_{n-1}\\
%0&0&\ddots&\ddots&\vdots\\
%0&0&\ddots&-\gamma_{n}&m_1\\
%\end{array}\right |}=\frac{1}{m_1+\lim_{n\to\infty}\gamma_1\frac{|U^rW^{2,n}|}{|W^{2,n}|}}=\frac{1}{m_1+\gamma_1\frac{m_2+\zeta'\frac{m_3\cdots}{m_1\cdots}}{m_1+\zeta'\frac{m_2\cdots}{m_1\cdots}}}.\end{equation}

%\begin{lemma}
%$\gamma_1L_1$ is decreasing strictly in $\beta_1$ and increasing strictly in every $\beta_i, i\geq 2.$
%\end{lemma}
%\begin{proof}
%Strict monotonicity on $\gamma_1$ can be proved using the last term of (\ref{rr}).
%Regarding (\ref{recurnu}), smaller $\beta_i$ for $i\geq 2$ implies stronger selection and bigger $\int y\nu(dy)_{0}$. % Strictly monotonicity on other $\gamma_i$ follows from (\ref{recurnu}) and 
%Since $\frac{1-\beta_{1}}{\int y\nu(dy)_{0}}=\gamma_{1}L_{1}$, $\gamma_{1}L_{1}$ is strictly increasing on each $\beta_i,i\geq 2.$  
%\end{proof}

\subsection{Appendix B}

\begin{lemma}\label{xy}
Let $X^n=(x_0^n,\cdots, x_{n}^n)$ be the unique solution of the equation
\begin{equation}\label{X}X^nW^n=r_1U^rW^n=(m_2,m_3,\cdots,m_{n+1}, m_{n+2}).\end{equation}
%Let $Y^n=(y_0^n,\cdots, y_{n}^n)$ be the unique solution of the equation
%\begin{equation}\label{Y}Y^nU^cW^n=R_1U^rU^cW^n=(m_2,m_3,\cdots, m_{n+1}, m_{n+3}).\end{equation}
Then $m_1<x_0^n<x_0^{n+1}<1$ for any $n\geq 1$. %;$\quad for any $1\leq i\leq n$, if $\gamma_i=\infty,$ then $x_i^n=y_i^n=0,$ otherwise, $y_i^n>x_i^n>0.$ 
%If $X, Y$ are vectors as in (\ref{X}) and (\ref{Y}), then both $X,Y$ are strictly positive. 
\end{lemma}
\begin{proof}
By Cramer's rule and Lemma \ref{gk1}
$$x_0^n=\frac{|U^rW^n|}{|W^n|}=R^n_{1}\in(m_1,1),\quad x_0^{n+1}=R^{n+1}_{1}\in(m_1,1).$$
%If $k=1$, then $X^1=(\frac{m_1m_2+\gamma_1m_3}{m_1^2+\gamma_1m_2},\frac{m_1m_3-m_2^2}{m_1^2+\gamma_1m_2})$. So if $\gamma_1\neq \infty,$ $x_0^1>0, x_1^1>0$, otherwise $x_0^1>0,x_1^1=0.$
%For $k=n$, assume that $x_0^n>0;$ for any $1\leq i\leq n$, $x_i^n>0$ if $\gamma_i< \infty$ and $ x_i^n=0$ otherwise. 
%Now let $k=n+1.$ 
For any $n\geq 1$, we are going to construct $X^{n+1}$ from $X^n$ and compare $x_0^n, x_0^{n+1}.$  The main argument is H\"older's inequality (\ref{holder}).

%We next prove that 
%for any $1\leq i\leq n$, if $\gamma_i=\infty,$ then $x_i^n=y_i^n=0,$ otherwise, $y_i^n>0, x_i^n>0.$  It suffices to study $x_i^n$, as the same approach works for $y_i^n.$

%The same strategy applies to $X$ and $Y$ and we could just focus on $X$. 
%To emphasis the dependence on $n$, we shall write 
%$$X=X^n=(x^n_1, \cdots, x^n_{n+1}).$$
%We use the induction. 
%For $n=1$, $X^1=(\frac{m_1m_2+\gamma_1m_3}{m_1^2+\gamma_1m_2},\frac{m_1m_3-m_2^2}{m_1^2+\gamma_1m_2})$. So if $\gamma_1\neq \infty,$ $x_0^1>0, x_1^1>0$, otherwise $x_0^1>0,x_1^1=0.$
%Assume that $X^n$ is exactly the same as described in the lemma. 
Note that 
$$x_0^nm_{n+1}+\cdots+x_{n}^nm_1=m_{n+2}.$$
Using (\ref{holder}), we get
\begin{equation}\label{<3}x_0^nm_{n+2}+\cdots+x_{n}^nm_2<m_{n+3}.\end{equation}
For $\varepsilon\geq 0$, let $x_0^{n,\varepsilon}=x_0^n+\varepsilon.$ Let $C^n$ be the matrix of $W^{n}$ with the last 
column removed. Then there exists a unique vector $X^{n,\varepsilon}=(x_0^{n,\varepsilon},\cdots, x_{n}^{n,\varepsilon})$ for a given $\epsilon$ such that 
\begin{equation}\label{xc}X^{n,\varepsilon}C^n=(m_2,m_3,\cdots,m_{n+1}).\end{equation}
It is clear that if $\gamma_i=\infty$, then $x_i^{n,\varepsilon}=0$; otherwise $x_i^{n,\varepsilon}$ is continuous and strictly increasing on $\varepsilon$.

To construct $X^{n+1}$ from $X^n$, the idea is to find a number $A_\epsilon\geq 0$ such that 
$$Y=(x_0^{n,\varepsilon},\cdots, x_{n}^{n,\varepsilon}, A_\epsilon)$$
satisfies $$YW^{n+1}=r_1U^rW^{n+1}=(m_2,m_3,\cdots,m_{n+1}, m_{n+2},m_{n+3}).$$
Then $X^{n+1}=Y.$

To achieve  this, let 
$$A_{\varepsilon}=\gamma_{n+1}^{-1}(x_0^{n,\varepsilon}m_{n+1}+\cdots+x_{n}^{n,\varepsilon}m_1-m_{n+2})(\equiv 0, \text{ if }\gamma_{n+1}=\infty).$$
Then the dot product of $Y$
and the second last column of $W^{n+1}$ gives $m_{n+2}$: 
$$x_0^{n,\varepsilon}m_{n+1}+\cdots+x_n^{n,\varepsilon}m_1-\gamma_{n+1}A_{\varepsilon}=m_{n+2}.$$
If $A_{\varepsilon}\not\equiv 0,$ then $A_{\varepsilon}$ is continuous and strictly increasing on $\varepsilon$ with $A_0=0$. Therefore, in view of (\ref{<3}), there
exists a unique $\varepsilon>0$ such that the dot product of $Y$
and the last column of $W^{n+1}$ gives $m_{n+3}$:
$$x_0^{n,\varepsilon}m_{n+2}+\cdots+x_n^{n,\varepsilon}m_2+A_{\varepsilon}m_1=m_{n+3}.$$
Then together with (\ref{xc}), %strictly positive satisfies 
$$YW^{n+1}=(m_2,m_3,\cdots,m_{n+3}).$$
So $X^{n+1}=Y.$ As $x_0^{n,\epsilon}$ is strictly increasing in $\epsilon$ and the $\epsilon$ in the above equality is strictly positive, we obtain that $0<x_0^n<x_0^{n,\epsilon}=x_0^{n+1}<1$.%;  for  $1\leq i\leq n$, $x_i^n<x_i^{n+1}$ if $\gamma_i< \infty;$  $x_i^n=x_i^{n+1}=0$ if $\gamma_i= \infty;$ for $i=n+1,$ $x_{n+1}^{n+1}>0$ if $\gamma_i< \infty$, 
% $x_{n+1}^{n+1}=0$ if $\gamma_i=\infty$. By induction, $0<x_0^n<x_0^{n+1}$ for any $n\geq 1.$ 

%Notice that 

%$$x_0^{n+1}= \frac{|U^rA^{n+1}|}{|A^{n+1}|}=
%\frac{|U^cA^{n}|}{|U^rU^cA^{n}|}=\frac{m_1 |U^rA^{n} |+\gamma_{n+1} |U^rU^cA^{n}|}{m_1\Big |A^{n} |+\gamma_{n+1} |U^cA^{n}|}>x_0^n=\frac{|U^rA^{n}|}{|A^{n}|}.$$

%So $$\frac{|U^rU^cA^{n}|}{|U^cA^{n}|}>x_0^n.$$

%But the left-hand term of the above display equals exactly $y_0^n.$ Then the lemma is proved. 
\end{proof}

\subsection{Appendix C}
\begin{proposition}\label{wij}For any $1\leq j<i\leq n$ and $b_i,b_j\in (0,1)$, $$\frac{\partial^2\ln |W^n|}{\partial b_i\partial b_j}>0, \,\,\forall n\geq i;\quad  \lim_{n\to\infty}\frac{\partial^2\ln |W^n|}{\partial b_i\partial b_j}>0.$$
\end{proposition}
\begin{proof}

Notice that 
$$|W^n|=\gamma_i\frac{d|W^n|}{d\gamma_i}+\left |\begin{array}{cc}
W^{1,i-1}&0\\
0&W^{i+1, n}\\
\end{array}\right |=\gamma_i\frac{d|W^n|}{d\gamma_i}+|W^{i-1}||W^{i+1,n} |.$$
Dividing both sides by $|W^n|$ yields
\begin{equation}\label{wgamma}1=\gamma_i\frac{d|W^n|}{d\gamma_i}/|W^n|+|W^{i-1} |\frac{|W^{i+1,n}|}{|W^n|}\end{equation}
Using the above display
\begin{align*}
\frac{\partial^2\ln |W^n|}{\partial b_i\partial b_j}&=-\frac{1}{b_i^2}\frac{\partial}{\partial b_j}(\frac{\partial |W^n|}{\partial \gamma_i}/|W^n|)&\\
&=-\gamma_i^{-1}\frac{1}{b_i^2}\frac{\partial}{\partial b_j}(1-|W^{i-1}||W^{i+1,n} |/|W^n|)&\\\
&=\gamma_i^{-1}\frac{1}{b_i^2}|W^{i+1,n} |\frac{\partial}{\partial b_j}(|W^{i-1} |/|W^n|)&\\
&=\gamma_i^{-1}\gamma_j^{-1}\frac{1}{b_i^2b_j^2}|W^{i+1,n}||W^{j-1} |/|W^n|^2\Big(|W^{n}||W^{j+1,i-1} |-|W^{i-1}||W^{j+1,n} |\Big)&\\
&=\gamma_i^{-1}\gamma_j^{-1}\frac{1}{b_i^2b_j^2}\frac{|W^{i+1,n}||W^{j-1}||W^{i-1}|}{|W^n|}\Big(\frac{|W^{j+1,i-1}|}{|W^{i-1}|}-\frac{|W^{j+1,n}|}{|W^n|}\Big)&\\
&=\frac{1}{(1-b_j)^2(1-b_i)^2}\gamma_1L_{1,n}\cdots\gamma_iL_{i,n}\frac{|W^{j-1}|}{\gamma_1\cdots\gamma_{j-1}}\frac{|W^{i-1}|}{\gamma_1\cdots\gamma_{i-1}}&\\
&\quad \times (\gamma_1L_{1,i-1}\cdots\gamma_jL_{j,i-1}-\gamma_1L_{1,n}\cdots\gamma_jL_{j,n}).&
\end{align*}
By Lemma \ref{rl}, we can conclude $\frac{\partial^2\ln |W^n|}{\partial b_i\partial b_j}>0$. Letting $n\to\infty$ we get the following% using the same arguments 
\begin{align*}
\lim_{n\to\infty}\frac{\partial^2\ln |W^n|}{\partial b_i\partial b_j}&=\frac{1}{(1-b_j)^2(1-b_i)^2}\gamma_1L_1\cdots\gamma_iL_i\frac{|W^{j-1}|}{\gamma_1\cdots\gamma_{j-1}}\frac{|W^{i-1}|}{\gamma_1\cdots\gamma_{i-1}}&\\
&\quad \times (\gamma_1L_{1,i-1}\cdots\gamma_jL_{j,i-1}-\gamma_1L_1\cdots\gamma_jL_j)>0.&
\end{align*}
%where the inequality is due to Remark \ref{llimit} and Remark \ref{gammal}. 
\end{proof}

\begin{corollary}\label{pushc}
For any $i\geq 1$, $\gamma_iL_i$ is strictly decreasing in $b_i$ and strictly increasing in $b_{j}, \,\, \forall j>i.$ The same result holds for $\gamma_iL_{i,n}$.
\end{corollary}
\begin{proof}
We shall only consider $\gamma_1L_1.$ The strict monotonicity in $b_1$ stems from (\ref{rjlj}). Take $j>1$. By (\ref{rjlj}), the monotonicity of $\gamma_1L_1$ in $b_j$ does not depend on $b_1$. For convenience let $b_1=c\in(0,1)$. Then we can study $L_1$ instead. Note that
\begin{align*}\frac{\partial L_1}{\partial b_j}=\lim_{n\to\infty}\frac{\partial L_{1,n}}{\partial b_j}&=\lim_{n\to\infty}\frac{|W^{2,n}|}{|W^{n}|}\Big(\frac{\partial |W^{2,n}|}{\partial b_j}/|W^{2,n} |-\frac{\partial |W^{n}|}{\partial b_j}/|W^n|\Big)&\\
&=L_1\lim_{n\to\infty}\Big(\frac{\partial |W^{2,n}|}{\partial b_j}/|W^{2,n} |-\frac{\partial |W^{n}|}{\partial b_j}/|W^n|\Big).&\end{align*}
Notice that the following holds when $b_1=1$, 
$$\frac{\partial |W^{n}|}{\partial b_j}/|W^n|=\frac{\partial |W^{2,n}|}{\partial b_j}/|W^{2,n} |.$$ 
Then by Proposition \ref{wij}
\begin{align*}
\lim_{n\to\infty}\Big(\frac{\partial |W^{2,n}|}{\partial b_j}/|W^{2,n} |-\frac{\partial |W^{n}|}{\partial b_j}/|W^n|\Big)&=\lim_{n\to\infty}\int_{c}^1 \frac{\partial }{\partial b_1}\Big(\frac{\partial |W^n|}{\partial b_j}/|W^n|\Big)db_1&\\
&=\lim_{n\to\infty}\int_c^1\frac{\partial^2 \ln |W^n|}{\partial b_1\partial b_j}db_1>0.&\end{align*}
Then we obtain $\frac{\partial L_1}{\partial b_j}>0.$
%If $\gamma_j=\infty$ (i.e., $b_j=0$), then using Proposition \ref{wij} and  (\ref{rjlj})
%\begin{align*}
%\frac{\partial L_1}{\partial b_j}\Big |_{b_j=0}=\lim_{b_j\to0}\frac{\partial L_1}{\partial b_j}&=\lim_{b_j\to 0}\lim_{n\to\infty}L_1\int_c^1\frac{\partial^2 \ln |W^n|}{\partial b_1\partial b_j}db_1&\\
%&=L_1\lim_{n\to\infty}\int_c^1\left(\frac{\partial^2 \ln |W^n|}{\partial b_1\partial b_j}\Big |%_{b_j=0}\right)db_1>0.&\end{align*}
\end{proof}

\begin{corollary}\label{uk}
For any $k>1,$ both $R_{1,k}^n$ and $R_{1,k}$ strictly decrease in $b_j$, for any $j\geq 1.$
\end{corollary}
\begin{proof}We shall prove only for $R_{1,k}$. 
Without loss of generality, we show that $R_{k+1,k}$ strictly decreases in $b_m, m\geq k+1.$ 
Take $\frac{|W^n|}{|W^n|}$ and expand the top $W^{n}$ for the first $k$ elements on the first row. A similar approach was used in obtaining (\ref{aanx}) where the expansion was made on the whole first row. Letting $n$ go to infinity we obtain the following, with detailed steps omitted
\begin{equation}\label{finitedev}1=(\prod_{j=0}^{k-1}\gamma_{1+j}L_{1+j})R_{k+1,k}+ \sum_{i=1}^{k-1}\Phi_{1,i}.\end{equation}
Taking derivative on $b_m$ on both sides, and using Corollary \ref{pushc}, the derivative of $R_{k+1,k}$ on $b_m$ is strictly negative for $b_m\in (0,1)$. 
\end{proof}

\subsection{Appendix D}

We introduce below a new notation for the special structure of matrix $W^n$. 
\begin{definition}\label{mij}
Assume $M$ is a square matrix of size $n$. For any $1\leq i\leq j\leq n,$ let $M(i,j)$ be the square matrix with $M_{i,i},M_{i,j},M_{j,i},M_{j,j}$ as the 4 corner elements. We say $M$ is of type $(*)$ if the following holds: $M_{i,j}>0$ if $i\leq j$; $M_{i,j}<0$ if $i=1+j$; 
$M_{i,j}=0$ if $i>1+j$. 
\end{definition}
By definition, $W^n$ is of type $(*)$. To compute the determinant of a matrix of type $(*)$, we need some more notations. Define
$$\EE_k^n:=\{e=(e_1,\cdots,e_k): 1=e_1<e_2<\cdots<e_k=n+1\}, \quad \forall \,2\leq k\leq n+1.$$
So $\EE_k^n$ consists of all sequences of length $k$ increasing from $1$ to $n+1$. 
Let $$\EE^n:=\cup_{2\leq k\leq n+1} \EE_k^n. $$
For $M$ of type $(*)$ and size $n$, define 
$$d(M):=M_{1,n}\prod_{i=2}^n|M_{i,i-1} |; \quad d_{M}(e):=\prod_{i=1}^{k-1}d(M(e_i,e_{i+1}-1)),\quad  \forall e\in\EE_k, \,\,2\leq k\leq n+1.$$
Let $s_n$ be the set of permutations of $\{1,2,\cdots,n\}.$
\begin{lemma}\label{type*}
For any matrix $M$ of type $(*)$ and of size $n$,
\begin{equation}\label{M||} |M |=\sum_{e\in \EE^n}d_M(e).\end{equation}
\end{lemma}
\begin{proof}
By decomposing $M$ along the last row, we can prove it by induction. Details are omitted. 
\end{proof}
\begin{remark}
Leibniz formula says that $|M|=\sum_{\sigma\in s_n}sgn(\sigma)\prod_{j=1}^nM_{j,\sigma(j)}$.  It is easy to see that the set $\{\sigma: \sigma\in s_n, \prod_{j=1}^nM_{j,\sigma(j)}\neq 0\}$ is in one-to-one correspondence to $\EE^n$. Moreover $sgn(\sigma)=1$ for any $\sigma$ in the former set. If we use $\sigma^e$ to denote the corresponding element in $s_n$ of an $e\in \EE_k,$ 
$$\prod_{j=1}^{k-1}d(M(e_j,e_{j+1}-1))=\prod_{j=1}^nM_{j,\sigma^e(j)}>0.$$
In other words, (\ref{M||}) is another writing of Leibniz formula. 
\end{remark}
We admit the following corollary with proof omitted. 
\begin{corollary}\label{ee}
$$|M(1,j)||M(j+1,n)|=\sum_{e\in\EE^{n+1}, \,\,j+1\in e}d_M(e).$$
\end{corollary}

\begin{lemma}\label{aga}
For any $1\leq j\leq n$ and $\gamma_j\in(0,\infty)$,  
$$\frac{d|W^n|}{d\gamma_j}/|W^n|\in(b_j,\frac{1}{\gamma_j}).$$%; \quad \lim_{n\to\infty}\frac{d|W^n|}{d\gamma_i}/|W^n|\in (b_i, \frac{1}{\gamma_i}).$$
\end{lemma}
\begin{proof}
By (\ref{wgamma}),  
$$\frac{d|W^n|}{d\gamma_j}/|W^n|=\gamma_j^{-1}\left(1-|W^{j-1} |\frac{|W^{j+1,n}|}{|W^n|}\right)=\gamma_j^{-1}\left(1-|W^{j-1} |\prod_{i=1}^jL_{i,n}\right).$$
Note that as long as $\gamma_j\neq \infty,$ we have $|W^{i-j-1} |\frac{|W^{j+1,n}|}{|W^n|}\in (0,1).$
Therefore 
$$\frac{d|W^n|}{d\gamma_j}/|W^n|<\gamma_j^{-1}.$$
%Similarly,
%\begin{equation}\label{laga}\lim_{n\to\infty}\frac{d|W^n|}{d\gamma_i}/|W^n|=\gamma_i^{-1}\left(1-|W^{i-1} |\prod_{j=1}^iL_j\right)<\gamma_i^{-1}.\end{equation}
To prove the strict lower bounds, using again (\ref{wgamma}), we just need to show  that
\begin{equation}\label{<1n}|W^{j-1}||W^{j+1, n} |/\frac{d|W^n|}{d\gamma_j}<1.\end{equation}
%and 
%\begin{equation}\label{<1}\lim_{n\to\infty}|W^{i-1}||W^{i+1, n} |/\frac{d|W^n|}{d\gamma_i}=\gamma_i(1-|W^{i-1} |\prod_{j=1}^iL_j)^{-1}-\gamma_i<1.\end{equation}
%=
%Apparently the matrix $\left(\begin{array}{cc}
%A^{1,i-1}&0\\
%0&A^{i+1, n}\\
%\end{array}\right)$ is not of type (*), but any of the two sub-matrices is. 
%To compute the determinant using (\ref{M||}), we have to require that $i+1\in l.$
Let $M$ be the matrix obtained by deleting the 
row and column of $W^n$ containing $\gamma_j$. Then 
$$|M|=\frac{d|W^n|}{d\gamma_j}.$$ 
The purpose is to compare $|W^{j-1}||W^{j+1,n}|$ and $|M|$. 
Denote 
$$A=\{e\in\EE^{n+1}:  j+1\in e\}.$$
Corollary \ref{ee} tells that
\begin{equation}\label{c9}|W^{j-1}||W^{j+1,n}|=\sum_{e\in A}d_{W^n}(e).\end{equation}
To compute $|M|$, we also seek to find an expression similar to the above display. Let $t(e)$ be the corresponding location such that $e_{t(e)}=j+1$ for any $e\in A$. Denote $$A'=\left\{e'\in \EE^{n}:  \exists e\in A, s.t.,
             \Bigg\langle\begin{array}{cc}
  e'_j=e_j, &  \mbox{if $i\leq t(e)-1$};\\  
             & \\  
e'_j=e_{j+1}-1, &    \mbox{if $j\geq t(e)\}$}.
             \end{array}  
. \right\}$$ There is a clear one-to-one correspondence between $A$ and $B$. It is easy to verify that 
$$|M|=\sum_{e'\in A'}d_{M}(e').$$ 
Consequently 
\begin{equation}\label{dd}|W^{j-1}||W^{j+1, n} |/|M|=\frac{\sum_{e\in A}d_{W^n}(e)}{\sum_{e'\in A'}d_{M}(e')}.\end{equation}
 Let $e\in A\cap \EE_k^{n+1}$ and $e'$ its corresponding element in $A'.$ Recalling the Definition \ref{mij},
\begin{align*}
d_{W^n}(e)&=d\Big(W^n(e_{t(e)-1},j)\Big)d\Big(W^n(j+1,e_{t(e)+1}-1)\Big)\prod_{i=1, i\notin\{ t(e)-1, t(e)\}  }^{k-1}d\Big(W^n(e_i,e_{i+1}-1)\Big)&\nonumber\\
&=\left(\prod_{i=e_{t(e)-1}, i\neq j }^{e_{t(e)+1}-2}\gamma_{i}\right)m_{j-e_{t(e)-1}+1}m_{e_{t(e)+1}-j-1}\prod_{i=1, i\notin\{ t(e)-1, t(e)\}  }^{k-1}d\Big(W^n(e_i,e_{i+1}-1)\Big)&
\end{align*}
and 
\begin{align*}&d_{M}(e')=\left(\prod_{i=e_{t(e)-1}, i\neq j }^{e_{t(e)+1}-2}\gamma_{i}\right)m_{e_{t(e)+1}-e_{t(e)-1}}\prod_{i=1, i\notin\{ t(e)-1, t(e)\}  }^{k-1}d\Big(W^n(e_i,e_{i+1}-1)\Big).&
\end{align*}
By H\"older's inequality (\ref{holder}), $$m_{j-e_{t(e)-1}+1}m_{e_{t(e)+1}-j-1}<m_{e_{t(e)+1}-e_{t(e)-1}}$$ Then \begin{equation}\label{d/d}\frac{d_{W^n}(e)}{d_M(e')}<1.\end{equation} So (\ref{<1n}) is proved. 

\end{proof}
\section{Acknowledgment}
The author thanks Takis Konstantopoulos, G\"otz Kersting and Pascal Grange for discussions. The author acknowledges the support of the National Natural Science Foundation of China (Youth Program, Grant:
11801458), and the XJTLU RDF-17-01-39.

\end{document}